\documentclass[11pt]{article}
\usepackage[titletoc,title]{appendix}
\usepackage{amsmath,amsfonts,amssymb,bm,amsthm,mathtools}
\usepackage[colorlinks,citecolor=blue,linkcolor=blue,pagebackref]{hyperref}
\usepackage{tikz}
\usetikzlibrary{shapes,arrows}
\usetikzlibrary{intersections,shapes.arrows}
\usetikzlibrary{decorations.pathreplacing}
\usetikzlibrary{arrows.meta}
\numberwithin{equation}{section}
\newtheorem{theorem}{Theorem}[section]
\newtheorem{lemma}[theorem]{Lemma}

\newtheorem{corollary}[theorem]{Corollary}

\theoremstyle{definition}

\newtheorem{definition}[theorem]{Definition}

\theoremstyle{remark}
\newtheorem{remark}[theorem]{Remark}

\usepackage{multicol}
\usepackage{longtable}

\usepackage[T1]{fontenc}
\usepackage[top=1in, bottom=1in, left=1in, right=1in]{geometry}
\usepackage{fancyhdr}
\newcommand{\dr}{\mathrm{d}}

\newcommand{\Ric}{\operatorname{Ric}}
\newcommand{\oRic}{\mathring{\mathrm{R}}\mathrm{ic}}
\newcommand{\Riem}{\operatorname{Riem}}
\newcommand{\Sc}{\operatorname{Sc}}
\newcommand{\curl}{\operatorname{curl}}
\newcommand{\prin}{\mathrm{prin}}

\DeclareMathOperator*{\Res}{Res}
\usepackage{enumerate}
\usepackage{xcolor,cancel}
\usepackage{relsize}

\renewcommand{\tilde}{\widetilde}

\usepackage{orcidlink}

\allowdisplaybreaks

\usepackage{pifont}
\renewcommand*{\backrefalt}[4]{%
\ifcase #1 %
No citations%
\or
\ding{43}~p.~#2%
\else
\ding{43}~pp.~#2%
\fi}

\begin{document}

\title{Higher order Weyl coefficients for the operator curl}

\author{
Giovanni Bracchi\,\orcidlink{0009-0005-0181-7092}
\thanks{GB:
Department of Mathematics,
University College London,
Gower Street,
London WC1E~6BT,
UK;
giovanni.bracchi.23@ucl.ac.uk.
}
\and
Matteo Capoferri\,\orcidlink{0000-0001-6226-1407}\thanks{MC:
Maxwell Institute for Mathematical Sciences
\&
Department of Mathematics,
Heriot-Watt University,
Edinburgh EH14 4AS, UK
\textit{and}
Dipartimento di Matematica ``Federigo Enriques'', Università degli Studi di Milano, Via C.~Saldini 50, 20133 Milano, Italy;
matteo.capoferri@unimi.it,
\url{https://mcapoferri.com}.
}
\and
Dmitri Vassiliev\,\orcidlink{0000-0001-5150-9083}
\thanks{DV:
Department of Mathematics,
University College London,
Gower Street,
London WC1E~6BT,
UK;
D.Vassiliev@ucl.ac.uk,
\url{http://www.homepages.ucl.ac.uk/\~ucahdva/}.
}}

\renewcommand\footnotemark{}

\date{3 July 2026}

\maketitle

\vspace{-.6cm}

\begin{abstract}
We establish refined spectral asymptotics for the operator curl acting on a connected oriented closed Riemannian 3-manifold. We treat positive and negative eigenvalues separately and obtain explicit formulae for the first six global Weyl coefficients. With local Weyl coefficients we compute the first four coefficients as well as the sixth one, and we determine the fifth up to a universal constant.
As a consequence, we prove that the eta function of curl (both in its local and global versions) is holomorphic in the complex half-plane $\operatorname{Re}s>-2$. Finally, under appropriate assumptions on the geodesic flow, we improve B\"ar's asymptotic formulae for the positive and negative counting functions, refining the remainder to $o(\lambda^2)$.

\

{\bf Keywords:} curl, spectral asymmetry, eta function, eta invariant, Weyl coefficients, spectral asymptotics.

\

{\bf 2020 MSC classes: }
primary
58J50; 
secondary
35P20, 
35Q61, 
47B93, 
47F99; 
58J28, 
58J40. 

\end{abstract}

\tableofcontents

\allowdisplaybreaks

\section{Statement of the problem and main results}
\label{Statement of the problem and main results}

The operator $\curl$ is one of the key operators of mathematical physics, alongside the Laplace--Beltrami and Dirac operators. Arguably, $\curl$ is the most fundamental of all, in that it is first-order as a differential operator and its definition requires a minimal number of geometric ingredients. 
In the 3-dimensional setting of classical physics, $\curl$ features prominently at the core of Maxwell's equations, as well as in hydrodynamics and magnetohydrodynamics, where its eigenfields --- known as \emph{Beltrami fields} --- are of physical significance.
In the context of Maxwell-type operators, $\curl$ has been studied from a variety of points of view, especially on bounded domains with physically meaningful boundary conditions --- see, e.g., \cite{safarov_curl,birman_curl,birman_curl2,lerner,
filonov_curl1,filonov_curl3,filonov_curl2}.

However, when it comes to the spectral geometry of $\curl$, the attention of the mathematical community has gained momentum only quite recently.
If we exclude the explicit calculation of the positive spectrum of $\curl$ on a Berger sphere performed by Lotay in a totally different context \cite{lotay}, the first significant contribution in this direction is due to B\"ar \cite{baer_curl}, who considered the $\curl$ operator on odd-dimensional closed oriented Riemannian manifolds, established basic spectral properties, obtained one-term Weyl asymptotics with sharp remainder, and computed the spectrum in several model cases. Another interesting recent line of enquiry is concerned with the phenomenon of spectral asymmetry, namely, the fact that in general positive and negative eigenvalues of the operator at hand are not mirror images of one another. For the special case of 3-dimensional closed manifolds, in \cite{curl, conjectures} the second and third author developed a new approach, analytic in flavour, to the study of spectral asymmetry for $\curl$, based on the use of pseudodifferential projections \cite{part1}. This approach is not specific to $\curl$, and can be deployed in a variety of scenarios, see, e.g., \cite{dirac_asymmetry}. Furthermore, there have been works on variational problems related to low eigenvalues of $\curl$, both in bounded domains and on closed Riemannian 3-manifolds \cite{peralta2, peralta1, peralta3}. Note that the spectral problem for $\curl$ on a closed manifold is different from spectral problems posited on manifolds with boundary or bounded domains. Indeed, the latter cannot be reduced to the former --- see \cite[Equations (1.3) and (1.4)]{curl} --- because physically meaningful boundary conditions prevent this.

\

We should like to emphasise that working with $\curl$ is more challenging than working with other fundamental operators of mathematical physics, in that $\curl$ is not elliptic. Furthermore, the fact that it acts on 1-forms, as opposed to scalar fields or spinors, makes the analysis trickier. In view of this, it is then not surprising that spectral geometric results on $\curl$ are few and far between, and the corresponding literature rather scant. Overall, working with $\curl$ is substantially different, both conceptually and technically, from scalar elliptic operators, and also different from Dirac-type operators, despite superficial similarities at the level of spectral asymptotics. Let us also observe that $\curl$ cannot be reduced to scalar operators by diagonalisation --- see, e.g., \cite{diagonalisation, obstructions}.

\

Within the context of the rich and noble history of research on spectral asymmetry, 
initiated by Atiyah, Patodi and Singer \cite{asymm1,asymm2,asymm3,asymm4},
our paper addresses the matter of detecting spectral asymmetry by looking at asymptotically large (in modulus) eigenvalues of the operator $\curl$, quantitatively improving and refining existing results, e.g.~\cite{baer_curl}, and providing explicit formulae. In particular, by carefully examining the $\curl$ propagator $e^{-it\curl}$, we establish for the first time the \emph{sixth} Weyl coefficients --- the lowest coefficients exhibiting asymmetry. 

\

Let $(M,g)$ be a connected oriented closed Riemannian manifold of dimension $d=3$.
We denote by $\Omega^k(M)=\Omega^k$ the Hilbert space of real-valued $k$-forms over $M$,
by $\dr$ the exterior derivative,
by $\delta$ the codifferential,
by $*$ the Hodge dual, and by $\rho=\rho(x)$ the Riemannian density.
Moreover, we denote by $\Riem$ the Riemann curvature tensor, by $\Ric$ the Ricci tensor, by $\Sc$ scalar curvature, and by
\begin{equation}
\label{trace-free Ricci}
\oRic_{\alpha\beta}:=\Ric_{\alpha\beta}-\frac13 \Sc \,g_{\alpha\beta}
\end{equation}
the trace-free Ricci tensor. Finally, we define 
\begin{equation}
\label{Levi-Civita tensor}
E_{\alpha\beta \gamma}(x):=\rho(x)\, \varepsilon_{\alpha\beta\gamma}
\end{equation}
to be the totally antisymmetric tensor, where $\varepsilon$ is the totally antisymmetric Levi-Civita symbol, $\varepsilon_{123}=1$.
Throughout the paper we adopt the differential geometric conventions
from~\cite[Appendix~A]{curl}, in particular, on the choice of sign of curvature.

\

This paper is concerned with the study of the operator $\curl$. As a differential expression $\curl$ is given by
\begin{equation*}
\label{curl as a differential expression}
\curl:=*\dr
\end{equation*}
acting on real-valued 1-forms.
The latter can be promoted to a self-adjoint operator
\begin{equation*}
\label{curl as an operator}
\curl:\delta\Omega^2\cap H^1\to\delta\Omega^2\,,
\end{equation*}
with discrete spectrum accumulating both to $+\infty$ and $-\infty$. Furthermore, zero is never an eigenvalue, see \cite{giga,curl}.

Let
\begin{equation}
\label{ev_curl}
\dots\le\lambda_{-2}\le\lambda_{-1}<0<\lambda_1\le\lambda_2\le\dots
\end{equation}
be the eigenvalues of $\curl$ enumerated in increasing order with account of multiplicity,
and let $u_j$, $j\in\mathbb{Z}\setminus\{0\}$, be the corresponding
normalised eigenforms.

Generically, the spectrum of $\curl$ is asymmetric about zero.
We study the positive and negative eigenvalues of $\curl$ separately
and define the two global counting functions as
\begin{equation}
\label{global counting functions for curl}
N_\pm(\lambda)
:=
\begin{cases}
0\quad&\text{for}\quad\lambda\le0,
\\
\sum_{0<\pm\lambda_j<\lambda}
\,1
\quad&\text{for}\quad\lambda>0,
\end{cases}
\end{equation}
and the two local counting functions as
\begin{equation}
\label{local counting functions for curl}
N_\pm(y;\lambda)
:=
\begin{cases}
0\quad&\text{for}\quad\lambda\le0,
\\
\sum_{0<\pm\lambda_j<\lambda}
\,*(u_j\wedge *u_j)(y)
\quad&\text{for}\quad\lambda>0,
\end{cases}
\qquad y\in M.
\end{equation}
Clearly, the local and global counting functions are related by the identity
\begin{equation*}
\label{relation between local and global counting functions}
N_\pm(\lambda)
=
\int_M
N_\pm(y;\lambda)\,\rho(y)\,\dr y\,,
\end{equation*}
where $\rho$ is the Riemannian density.

The goal of this paper is the study of the asymptotic behaviour of the counting functions
\eqref{global counting functions for curl}
and
\eqref{local counting functions for curl}
as $\lambda\to+\infty$.

Our first main result is

\begin{theorem}
\label{theorem 1}
We have
\begin{equation}
\label{theorem 1 equation 1}
N_\pm(y;\lambda)
=
\frac{1}{6\pi^2}\,\lambda^3
+
O(\lambda^2)
\qquad \text{as}\qquad \lambda \to + \infty
\end{equation}
with remainder uniform over $y\in M$.
\end{theorem}

Theorem~\ref{theorem 1} immediately implies

\begin{corollary}
\label{corollary 1}
We have
\begin{equation}
\label{corollary 1 equation 1}
N_\pm(\lambda)
=
\frac{\operatorname{Vol}M}{6\pi^2}\,\lambda^3
+
O(\lambda^2)
\qquad \text{as}\qquad \lambda \to + \infty
\,.
\end{equation}
\end{corollary}

Corollary~\ref{corollary 1} is due to B\"ar \cite{baer_curl}.
For a detailed proof see \cite{curl}.

Under appropriate geometric assumptions on the geodesic flow
the remainder $O(\lambda^2)$ in formulae
\eqref{theorem 1 equation 1}
and
\eqref{corollary 1 equation 1}
can be replaced by $o(\lambda^2)$. In what follows we formulate these geometric conditions.

Let $(x^{(+)}(t;y,\eta),\xi^{(+)}(t;y,\eta))$ be the geodesic flow
on $T^*M\setminus\{0\}$,
with $(y,\eta)$ being the starting point.
See formulae
\eqref{eigenvalues principal symbol curl}--\eqref{Hamilton's equations}
for details.
We say that we are looking at a \emph{geodesic loop} if
$x^{(+)}(T;y,\eta)=y$ for some $T>0$.
We say that we are looking at a \emph{periodic geodesic} if
$(x^{(+)}(T;y,\eta),\xi^{(+)}(T;y,\eta))=(y,\eta)$ for some $T>0$.

Let $S^*_yM:=\{\eta\ |\ \|\eta\|=1\}\subset T^*_yM$
be the unit cosphere at the point $y\in M$.
Denote by $\Pi_y$ the set of $\eta\in S^*_yM$
such that $(y,\eta)$ is a starting point for a geodesic loop.
The Riemannian metric on $M$ induces a natural Lebesgue measure on
$S^*_yM$ and it is known
\cite[Lemma~1.8.2]{SaVa}
that the set $\Pi_y$ is measurable.

Similarly, let $S^*M:=\{(y,\eta)\ |\ \|\eta\|=1\}\subset T^*M$
be the unit cosphere bundle.
Denote by $\Pi$ the set of $(y,\eta)\in S^*M$
which serve as starting points for periodic geodesics.
The Riemannian metric on $M$ induces a natural Lebesgue measure on
$S^*M$ and it is known
\cite[Lemma~1.3.4]{SaVa}
that the set $\Pi$ is measurable.

\begin{definition}
\label{Definition 1 of nonfocal}
A point $y\in M$ is said to be \emph{nonfocal}
if the set $\Pi_y$ has measure zero.
\end{definition}

\begin{definition}
\label{Definition 1 of nonperiodicity}
We say that the \emph{nonperiodicity condition} is fulfilled
if the set $\Pi$ has measure zero.
\end{definition}

The sets $\Pi_y$ and $\Pi$ may be large and proving that they have zero measure may not be easy.
This impediment can be overcome by working with smaller sets which have the same measure
as the original sets $\Pi_y$ and $\Pi$.

We call a loop of length $T>0$ \emph{absolutely focussed}
if distance squared between
$x^{(+)}(T;y,\eta)$ and $y$ has an infinite order zero as a function of $\eta$.
By $\Pi_y^\mathrm{a}$ we denote the set of $\eta\in\Pi_y$
such that $(y,\eta)$ is a starting point for an absolutely focussed geodesic loop.
It is known
\cite[Lemma~1.8.3]{SaVa}
that the set $\Pi_y^\mathrm{a}$ is measurable
and, moreover, the set $\Pi_y\setminus\Pi_y^\mathrm{a}$ has measure zero.
In other words, almost all loops are absolutely focussed.

Similarly, we call a $T$-periodic geodesic \emph{absolutely periodic}
if distance squared between
\linebreak
$(x^{(+)}(T;y,\eta),\xi^{(+)}(T;y,\eta))$ and $(y,\eta)$
has an infinite order zero as a function of $(y,\eta)$.
By $\Pi^\mathrm{a}$ we denote the set of $(y,\eta)\in\Pi$
such that $(y,\eta)$ is a starting point for an absolutely periodic geodesic.
It is known
\cite[Corollary~1.3.6]{SaVa}
that the set $\Pi^\mathrm{a}$ is measurable
and, moreover, the set $\Pi\setminus\Pi^\mathrm{a}$ has measure zero.
In other words, almost all periodic geodesics are absolutely periodic.

Definitions
\ref{Definition 1 of nonfocal}
and
\ref{Definition 1 of nonperiodicity}
can now be reformulated in the following more user-friendly manner.

\begin{definition}
\label{Definition 2 of nonfocal}
A point $y\in M$ is said to be \emph{nonfocal}
if the set $\Pi_y^\mathrm{a}$ has measure zero.
\end{definition}

\begin{definition}
\label{Definition 2 of nonperiodicity}
We say that the \emph{nonperiodicity condition} is fulfilled
if the set $\Pi^\mathrm{a}$ has measure zero.
\end{definition}

Note that if the Riemannian manifold $(M,g)$ is real analytic
Definitions
\ref{Definition 2 of nonfocal}
and
\ref{Definition 2 of nonperiodicity}
simplify further.
Namely, a point $y\in M$ is nonfocal if and only if there does not exist
an absolutely focussed geodesic loop emanating from $y$.
Similarly, the nonperiodicity condition is fulfilled if and only if there does not exist
an absolutely periodic geodesic.

Our second and third main results are the following two theorems.

\begin{theorem}
\label{Theorem main result 2}
Suppose the point $y\in M$ is nonfocal. Then the local counting functions admit the refined asymptotic expansion
\begin{equation}
\label{Theorem main result 2 equation 1}
N_\pm(y;\lambda)=\frac{1}{6\pi^2}\,\lambda^3 + o(\lambda^2) \qquad \text{as}\qquad \lambda \to + \infty\,.
\end{equation}
\end{theorem}

\begin{theorem}
\label{Theorem main result 3}
Suppose the nonperiodicity condition is satisfied. Then the global counting functions admit the refined asymptotic expansion
\begin{equation}
\label{Theorem main result 3 equation 1}
N_\pm(\lambda)=\frac{\operatorname{Vol}M}{6\pi^2}\,\lambda^3 + o(\lambda^2)
 \qquad \text{as}\qquad \lambda \to + \infty\,.
\end{equation}
\end{theorem}

It is known that refining asymptotic formulae
\eqref{Theorem main result 2 equation 1}
and
\eqref{Theorem main result 3 equation 1}
beyond $o(\lambda^2)$ is unfeasible due to obstacles of number-theoretic nature.
This leads us to the study of mollified counting functions.
It is more convenient to start addressing the mollification issue by examining first the mollified version not of the counting functions
$N_\pm(y;\lambda)$ and $N_\pm(\lambda)$ themselves,
but their derivatives
$N_\pm'(y;\lambda)$ and $N_\pm'(\lambda)$
with respect to $\lambda$. See also Remark~\ref{rmk_full_asymptotic_expansion}.

Let $\widehat{\mu}:\mathbb{R}\to \mathbb{R}$ be a smooth even function such that $\widehat{\mu}= 1$ in some neighbourhood of the origin and 
$\operatorname{supp} \widehat{\mu}\subset(-T_0,T_0)$,
where $T_0$ is the infimum of lengths of all the geodesic loops originating from all the points of the manifold. Let $\mu$ be the inverse Fourier transform of $\widehat{\mu}$.
For the Fourier transform and inverse Fourier transform we adopt the convention
\cite[formulae~(1.15)~and~(1.16)]{dirac}.
We have
\begin{equation}
\label{local counting function mollified}
(N_\pm'*\mu)(y;\lambda)
=
c^\pm_2(y)\,\lambda^2
+
c^\pm_1(y)\,\lambda
+
c^\pm_0(y)
+
c^\pm_{-1}(y)\,\lambda^{-1}
+\ldots,
\end{equation}
\begin{equation}
\label{global counting function mollified}
(N_\pm'*\mu)(\lambda)
=
c^\pm_2\,\lambda^2
+
c^\pm_1\,\lambda
+
c^\pm_0
+
c^\pm_{-1}\,\lambda^{-1}
+\ldots,
\end{equation}
as $\lambda\to+\infty$. Here the star stands for convolution in the variable $\lambda$.
We call the coefficients $c^\pm_k$ appearing in
\eqref{local counting function mollified}
and
\eqref{global counting function mollified}
\emph{Weyl coefficients}, local and global respectively.
Of course, the local and global Weyl coefficients are related by the identity
\begin{equation}
\label{relation between local and global Weyl coefficients}
c^\pm_k=\int_M c^\pm_k(y)\,\rho(y)\,\dr y\,.
\end{equation}
It is known \cite{baer_curl} that
$c^\pm_2=\frac{\operatorname{Vol}M}{2\pi^2}\,$.

\begin{remark}
\label{rmk_full_asymptotic_expansion}
The classical literature \cite{DuGu,Ivr80,Ivr84,Ivr98,SaVa} justifying the existence of full asymptotic expansions of the form \eqref{local counting function mollified} and \eqref{global counting function mollified} does not appropriately cover the case of $\curl$, the latter not being elliptic. However, the required result follows \emph{a posteriori} from our construction of the propagator for $\curl$ presented in Sections~\ref{The algorithm} and~\ref{Justification of the algorithm}.
\end{remark}

Weyl coefficients enjoy the following remarkable symmetry property.

\begin{theorem}
\label{theorem symmetries of Weyl coefficients}
We have
\begin{equation}
\label{Symmetries of Weyl coefficients equation}
c_k^+(y)
=
(-1)^k\,
c_k^-(y)
\end{equation}
for all $k=2,1,0,-1,-2,\dots\,$.
\end{theorem}

The statement of Theorem~\ref{theorem symmetries of Weyl coefficients} is part of the folklore in the subject, but we were unable to identify a rigorous proof for it in the literature, especially for $\curl$. Hence, for the sake of completeness and the benefit of the reader, we provide a proof in Appendix~\ref{appendix symmetry of Weyl coefficients}.

\

Our last main result is the following theorem, providing explicit formulae for the first six Weyl coefficients of $\curl$.

\begin{theorem}
\label{Theorem main result}
We have
\begin{align}
\label{Theorem main result equation 0}
c_2^\pm(y)
&=
\frac{1}{2\pi^2}\,,
\\[.5em]
\label{Theorem main result equation 1}
c_1^\pm(y)
&=
0\,,
\\[.5em]
\label{Theorem main result equation 2}
c_0^\pm(y)
&=
-\frac{1}{12\pi^2}\,\Sc(y)\,,
\\[.5em]
\label{Theorem main result equation 3}
c_{-1}^\pm(y)
&=
0\,,
\\[.5em]
\label{Theorem main result equation 4}
c_{-2}^\pm(y)
&=
-
\frac{1}{80\pi^2}
\|\oRic(y)\|^2
+
C\,
(\Delta\Sc)(y)\,,
\\[.5em]
\label{Theorem main result equation 5}
c_{-3}^\pm(y)
&=
\mp\,
\frac{1}{120\pi^2}\,
E^{\alpha\beta\gamma}(y)\,\oRic_{\alpha\mu}(y)\,\nabla_\beta\,\oRic_\gamma{}^\mu(y)\,,
\end{align}
where $C$ is some universal constant, the same for $c_{-2}^+$ and $c_{-2}^-$, and
\begin{equation}
\label{trace-free ricci squared}
\|\oRic(y)\|^2:=\oRic_{\alpha\beta}(y)\,
\oRic^{\alpha\beta}(y)\,.
\end{equation}
\end{theorem}

\begin{remark}
Let us emphasise that spectral asymmetry manifests itself in the asymptotics only at the level of the \emph{sixth} Weyl coefficients, see~\eqref{Theorem main result equation 5}. For this reason, it is in general easier to detect spectral asymmetry by examining low, as opposed to high, eigenvalues.
\end{remark}

\begin{remark}
In the current paper we do not determine the universal constant $C$ appearing in \eqref{Theorem main result equation 4}. However, we should like to point out that its precise value is irrelevant when it comes to computing \emph{global} Weyl coefficients, because the term $C\, (\Delta \Sc)(y)$ integrates to zero upon substitution into~\eqref{relation between local and global Weyl coefficients}.
\end{remark}

Integrating
\eqref{local counting function mollified}
and
\eqref{global counting function mollified}
from $-\infty$ to a given positive $\lambda$ with account of \eqref{Theorem main result equation 3}, we arrive at
\begin{corollary}
The mollified local and global counting functions admit the asymptotic expansions
\begin{equation}
\label{local counting function integrated}
(N_\pm*\mu)(y;\lambda)
=
\frac13 c^\pm_2(y)\,\lambda^3
+
\frac12 c^\pm_1(y)\,\lambda^2
+
c^\pm_0(y)\,\lambda
+
b^\pm(y)
-
c^\pm_{-2}(y)\,\lambda^{-1}
-
\frac12c^\pm_{-3}(y)\,\lambda^{-2}
+\ldots,
\end{equation}
\begin{equation}
\label{global counting function integrated}
(N_\pm'*\mu)(\lambda)
=
\frac13 c^\pm_2\,\lambda^3
+
\frac12 c^\pm_1\,\lambda^2
+
c^\pm_0\,\lambda
+
b^\pm
-
c^\pm_{-2}\,\lambda^{-1}
-
\frac12c^\pm_{-3}\,\lambda^{-2}
+\ldots,
\end{equation}
as $\lambda\to+\infty$.
\end{corollary}
\begin{remark}
The functions $b^\pm(y)$ and their integrals $b^\pm$ in \eqref{local counting function integrated} and~\eqref{global counting function integrated}, respectively, are due to contributions from small eigenvalues, and are not determined by means of microlocal techniques.
\end{remark}

There are many reasons why one may be interested in examining higher order Weyl coefficients. An important motivation is offered by their relation with the eta function. Let us make this precise.

We define the \emph{local} and \emph{global} eta functions of $\curl$ as
\begin{equation}
\label{local eta function}
\eta(y;s):=\sum_{k\in \mathbb{Z}\setminus\{0\}} \frac{\operatorname{sgn}\lambda_k}{|\lambda_k|^s}\, {*\!}\left(u_k(y)\wedge *u_k(y)\right), \qquad y\in M, \,s\in \mathbb{C}\,,
\end{equation}
and
\begin{equation}
\label{global eta function}
\eta(s):=\sum_{k\in \mathbb{Z}\setminus\{0\}} \frac{\operatorname{sgn}\lambda_k}{|\lambda_k|^s}\,,\qquad s\in \mathbb{C}\,,
\end{equation}
respectively. Of course, \eqref{global eta function} is obtained by integrating~\eqref{local eta function} over $M$. Both \eqref{global eta function} and \eqref{local eta function} converge absolutely for $\operatorname{Re}s>3$, and can be continued meromorphically to the whole complex plane, potentially with simple poles at integers smaller than or equal to $3$.
One can show \cite[Theorem~3.2.2]{fang_PhD} that the residues of the local eta function are given by
\begin{equation}
\label{residues related to Weyl coefficients}
\mathrm{Res}(\eta(y;\,\cdot\,),n)
=c_{n-1}^+(y)-c_{n-1}^-(y)\,,
\qquad
n=3,2,1,0,-1,-2,\ldots.
\end{equation}
In particular, Theorem~\ref{Theorem main result} implies 
\begin{equation*}
\mathrm{Res}(\eta(y;\,\cdot\,),n)
=0\,,
\qquad\text{for}\qquad
n=3,2,1,0,-1,
\end{equation*}
and
\begin{equation}
	\label{local eta residue -2}
	\mathrm{Res}(\eta(y;\,\cdot\,),-2) = - \frac{1}{60\pi^2}\,E_\alpha{}^{\beta\gamma}(y)\,\oRic^{\alpha\mu}(y)\,\nabla_\beta\,\oRic_{\gamma\mu}(y)\,.
	\end{equation}
That is, the first (potential) pole of the eta function of curl occurs at $s=-2$. Also observe that Theorem~\ref{theorem symmetries of Weyl coefficients} and formula~\eqref{residues related to Weyl coefficients} tell us that subsequent poles can only occur at negative \emph{even} integer values of $s$.

All in all, we have the following
\begin{corollary}
The eta functions of $\curl$, local and global, are holomorphic in the half-plane $\operatorname{Re}s>-2$.
\end{corollary}

Note that in~\cite[Theorem 4.14(iii)]{asymm2} the authors stated, without a proof, that the global eta function of $\curl$ is holomorphic in the half-plane $\operatorname{Re}s>-1/2$.

\subsection*{Structure of the paper}
\addcontentsline{toc}{subsection}{Structure of the paper}

Our paper is structured as follows.

In Section~\ref{The strategy for the proof} we outline, in plain English, the strategy for our main proofs. The latter relies on the construction of the propagator for $\curl$, whose constructive algorithm is set out in detail in Section~\ref{The algorithm}. Section~\ref{Justification of the algorithm} then puts such algorithm on a rigorous footing, providing a mathematical justification for it, whereas Section~\ref{Implementation of the algorithm} introduces some simplifications allowing for its implementation in practice.

Sections~\ref{The first two Weyl coefficients}, \ref{The third Weyl coefficient} and \ref{Higher Weyl coefficients} address the proof of the formulae for the first six Weyl coefficients for $\curl$, whereas Section~\ref{Proof of Theorems on Weyl asymptotics} contains the proof of our three main theorems on refined spectral asymptotics --- Theorems~\ref{theorem 1}, \ref{Theorem main result 2} and~\ref{Theorem main result 3}.
Finally, in Section~\ref{Examples} we present explicit examples which both inform and verify the main results.

The paper is complemented by three appendices providing complementary results and technical materials.

\newpage

\subsection*{List of notation}
\addcontentsline{toc}{subsection}{List of notation}

\begin{longtable}{l l}
\hline
\\ [-1em]
\multicolumn{1}{c}{\textbf{Symbol}} & 
  \multicolumn{1}{c}{\textbf{Description}} \\ \\ [-1em]
 \hline \hline \\ [-1em]
$\sim$ & Asymptotic expansion \\ \\ [-1em]
$\ast$ & Hodge dual \\ \\ [-1em]
$\|\,\cdot\,\|$ & Riemannian norm, \eqref{norm of xi}, \eqref{trace-free ricci squared} \\ \\ [-1em]
$|\,\cdot\,|$ & Euclidean norm \\ \\ [-1em]
$b^{(\pm)}(x,\xi)$ & Eigencovector of $\curl_{\prin}(x,\xi)$ associated to $h^{(\pm)}(x,\xi)$ \\ \\ [-1em]
$b^{(\pm)}_\shortparallel(t;y,\eta)$ & Parallel transport of $b^{(\pm)}(y,\eta)$ along $(x^{(\pm)},\xi^{(\pm)})$, Remark~\ref{gauge transformations now sit at the source} \\ \\ [-1em]
$c_n(y)$ & Local Weyl coefficients of $\sqrt{-\Delta}$ \\ \\ [-1em]
$c^\pm_n(y)$ & Positive and negative local Weyl coefficients of $\curl$ \\ \\ [-1em]
$\curl_E$ & Extended curl, \eqref{definition extended curl}  \\ \\ [-1em]
$\chi^{(\aleph)}$, $\quad\aleph\in\{+,-,0\}$ & Cut-offs, see \eqref{time-dependent oscillatory integrals}--\eqref{time-independent oscillatory integral} \\ \\ [-1em]
$d$ & Dimension of the manifold $M$\\ \\ [-1em]
$\dr$ & Exterior derivative \\ \\ [-1em]
$\delta$ & Codifferential \\ \\ [-1em]
$\Delta:=-\delta\dr$ & (Nonpositive) Laplace--Beltrami operator \\ \\ [-1em]
$\nabla$ & Covariant derivative\\ \\ [-1em]
$\dr\Omega^{k-1}(M)$ & Exact $k$-forms over $M$ \\ \\ [-1em]
$\delta\Omega^{k+1}(M)$ & Coexact $k$-forms over $M$ \\ \\ [-1em]
$E_{\alpha\beta\gamma}$ & Totally antisymmetric tensor \eqref{Levi-Civita tensor} \\ \\ [-1em]
$\varepsilon_{\alpha\beta\gamma}$ & Totally antisymmetric symbol, $\varepsilon_{123}=+1$ \\ \\ [-1em]
$F^{(\aleph)}_j$, $\quad\aleph\in\{+,-,0\}$ & The operator \eqref{operator F j} associated with $\varphi^{(\aleph)}$ \\ \\ [-1em]
$f_{x^\alpha}$ & Partial derivative of a function $f$ with respect to $x^\alpha$ \\ \\ [-1em]
$f^{(\pm)}_{-k}$ & Component of $v^{(\pm;\pm)}_{-k}$ along $b^{(\pm)}(x^{(\pm)},\xi^{(\pm)})$, \eqref{first transport equation ODE plus}, \eqref{first transport equation ODE plus adjustment} \\ \\ [-1em]
$\varphi^{(\pm)}(t,x;y,\eta)$ & Phase functions associated with $h^{(\pm)}$ , see \eqref{phase functions condition 1}--\eqref{phase functions condition 4} \\ \\ [-1em]
$\varphi^{(0)}(x;y,\eta)$ & Phase function at $t=0$, \eqref{two phase functions coincide at time zero} \\ \\ [-1em]
$g$ & Riemannian metric \\ \\ [-1em]
$g_a$ & Berger metric on the $3$-sphere with parameter $a$ \\ \\ [-1em]
$\Gamma^\alpha{}_{\beta\gamma}$ & Christoffel symbols\\ \\ [-1em]
$h^{(\aleph)}(x,\xi)$, $\quad\aleph\in\{+,-,0\}$ & Eigenvalues of $\curl_{\prin}(x,\xi)$, \eqref{eigenvalues principal symbol curl} \\ \\ [-1em]
$\mathcal{H}^k(M)$ & Harmonic $k$-forms over $M$ \\ \\ [-1em]
$\eta(y;s)$ & Local eta function of $\curl$ \\ \\ [-1em]
$\eta(s)$ & Eta function of $\curl$ \\ \\ [-1em]
$\operatorname{I}$ & Identity matrix \\ \\ [-1em]
$\operatorname{Id}$ & Identity operator \\ \\ [-1em]
$\operatorname{id}$ & Full right symbol of the identity operator\\ \\ [-1em]
$\operatorname{id}^{(\beth)}$, $\quad\beth\in\{+,-,0\}$ & Projection of $\operatorname{id}$ along $P^{(\beth)}$, \eqref{projected symbols identity} \\ \\ [-1em]
$L_\alpha^{(\aleph)}$, $\quad\aleph\in\{+,-,0\}$ & The operator \eqref{operator L alpha} associated with $\varphi^{(\aleph)}$ \\ \\ [-1em]
$(\lambda_j, u_j)$, $\quad j=\pm1, \pm2, \ldots$ & Eigensystem for $\curl$, see~\eqref{ev_curl} \\ \\ [-1em]
$\operatorname{mod} \ \Psi^{-\infty}$ & Modulo an integral operator with infinitely smooth kernel,\\ & \quad possibly depending smoothly on $t$ \\ \\ [-1em]
$\mu$ & Mollifier, see~\eqref{local counting function mollified} and~\eqref{global counting function mollified} \\ \\ [-1em]
$(\mu_j, f_j)$, $\quad j=0,1,2,\ldots$ & Eigensystem for $-\Delta$ \\ \\
[-1em]
$N_\pm(\lambda)$ & Positive and negative global counting functions of $\curl$ \eqref{global counting functions for curl} \\ \\ [-1em]
$N_\pm(y;\lambda)$ & Positive and negative local counting functions of $\curl$ \eqref{local counting functions for curl} \\ \\ [-1em]
$N_\mathrm{I}(y;\lambda)$, $N_\mathrm{II}(y;\lambda)$, $N_\mathrm{III}(y;\lambda)$ & Partial local counting functions of $\curl$ on the Berger sphere \\ \\ [-1em]
$\Omega^k(M)$ & Differential $k$-forms over $M$ \\ \\ [-1em]
$P_\pm$ & Projections onto the positive and negative spectrum of $\curl$
\\ \\ [-1em]
$P_0$ & Orthogonal projection onto $\dr\Omega^0(M)$
\\ \\ [-1em]
$P_{\mathcal{H}^1}$ & Orthogonal projection onto $\mathcal{H}^1(M)$ \\ \\ [-1em]
$P^{(\aleph)}$, $\quad\aleph\in\{+,-,0\}$ & Principal symbol of $P_\aleph$, \eqref{projection hermiticity}--\eqref{projection completeness} \\ \\ [-1em]
$P^{(\aleph;\beth)}$, $\quad\aleph,\beth\in\{+,-,0\}$ & Evaluation of $P^{(\aleph)}$ along $(x^{(\pm)},\xi^{(\pm)})$ or at \((y,\eta)\), \eqref{principal symbols of projection operators as functions of time}, \eqref{made this a displayed formula on 20 April 2026} \\ \\ [-1em]
$\Psi^s$ & Classical pseudodifferential operators of order $s$ \\ \\ [-1em]
$\Psi^{-\infty}$ & Infinitely smoothing operators \\ \\ [-1em]
$Q_\prin$ & Principal symbol of the pseudodifferential operator $Q$ \\ \\ [-1em]
$q^{(\pm)}_0$ & The functions \eqref{formula for q 4} \\ \\ [-1em]
$q^{(\pm)}_{0,\shortparallel}$ & The functions \eqref{formula for q 4 parallel} \\ \\ [-1em]
$\Ric$, $\oRic$ & Ricci tensor and trace-free Ricci tensor, \eqref{trace-free Ricci}\\ \\ [-1em]
$\Riem$ & Riemann curvature tensor  \\ \\ [-1em]
$\rho(x)$ & Riemannian density \\ \\ [-1em]
$\Sc$ & Scalar curvature  \\ \\ [-1em]
$S^*M$ & Cosphere bundle  \\ \\ [-1em]
$\mathfrak{S}^{(\aleph)}$, $\quad\aleph\in\{+,-,0\}$ & Amplitude-to-symbol operator associated with $\varphi^{(\aleph)}$, \eqref{22 April 2026 equation 1}--\eqref{S minus k} \\ \\ [-1em]
$U(t)$ & (Total) propagator of $\curl$, \eqref{propagator curl} \\ \\ [-1em]
$U_\pm(t)$ & Positive and negative propagators of $\curl$, \eqref{propagator curl plus}, \eqref{propagator curl minus} \\ \\ [-1em]
$\mathfrak{u}_\pm(t,x,y)$ & Schwartz kernel of $U_\pm(t)$ \\ \\ [-1em]
$\operatorname{Vol} M$, $\operatorname{Vol}_gM$ & Riemannian volume of $(M,g)$ \\ \\ [-1em]
$V(t)$, $V_\aleph$, $\quad\aleph\in\{+,-,0\}$& The auxiliary oscillatory integrals \eqref{auxiliary operator}, \eqref{time-dependent oscillatory integrals}, \eqref{time-independent oscillatory integral} \\ \\ [-1em]
$v^{(\aleph)}$, $\quad\aleph\in\{+,-,0\}$ & Full right symbol of $V_\aleph$ \\ \\ [-1em]
$v^{(\aleph;\beth)}$, $\quad\aleph,\beth\in\{+,-,0\}$ & Projection of $v^{(\aleph)}$ along $P^{(\aleph;\beth)}$, \eqref{projected symbols plus minus}, \eqref{projected symbols zero} \\ \\ [-1em]
$w^{(\aleph)}$, $\quad\aleph\in\{+,-,0\}$ & The amplitude \eqref{asymptotic expansiions for the w} \\ \\ [-1em]
$w^{(\aleph;\beth)}$, $\quad\aleph,\beth\in\{+,-,0\}$ & Projection of $w^{(\aleph)}$ along $P^{(\aleph;\beth)}$, \eqref{Introduce the rank 1 matrix-functions} \\ \\ [-1em]
$(x^{(\pm)},\xi^{(\pm)})$ & Hamiltonian flow of $h^{(\pm)}$ \\ \\ [-1em]
$\zeta_Q(s)$ & Zeta function of the operator $Q$ \\ \\ [-1em]
\hline
\end{longtable}

\section{The strategy for the proof}
\label{The strategy for the proof}

To begin with, let us outline the strategy for the proof of our main results.

\begin{definition}
We define the \emph{propagator for the operator $\curl$} as the time-dependent bounded operator in the Hilbert space $\delta\Omega^2$ given by the formula
\begin{equation}
\label{propagator curl}
U(t)
:=
e^{-it\curl}
=
\sum_{j\in\mathbb{Z}\setminus\{0\}}
e^{-i\lambda_j t\,}
u_j\,\langle u_j\,,\,\cdot\ \rangle\,.
\end{equation}
\end{definition}

Of course, the operator $U(t)$ is unitary. Having the propagator at one's disposal allows one to compute spectral asymptotics by means of known, classical techniques, see e.g.~\cite[Appendix~B]{SaVa}.

As our task is to study the distribution of positive and negative eigenvalues of $\curl$ separately, it is convenient to split the propagator $U(t)$ into a sum of two propagators $U(t)=U_+(t)+U_-(t)$,
where
\begin{equation}
\label{propagator curl plus}
U_+(t):=\sum_{j=1}^{+\infty}
e^{-i\lambda_j t}\,
u_j\,\langle u_j\,,\,\cdot\ \rangle\,,
\end{equation}
\begin{equation}
\label{propagator curl minus}
U_-(t):=\sum_{j=-1}^{-\infty}
e^{-i\lambda_j t}\,
u_j\,\langle u_j\,,\,\cdot\ \rangle\,.
\end{equation}

The operators $U(t)$, $U_+(t)$ and $U_-(t)$ were initially defined in the Hilbert space $\delta\Omega^2$, but they admit an obvious extension to the full Hilbert space $\Omega^1$ of square-integrable 1-forms: the expansions in the right-hand sides of formulae \eqref{propagator curl}--\eqref{propagator curl minus} make sense when acting on an arbitrary square-integrable 1-form. Of course, the extended operator $U(t):\Omega^1\to\Omega^1$ is no longer unitary.

The central idea underpinning our analysis is the introduction of the auxiliary operator
\begin{equation}
\label{auxiliary operator}
V(t)
:=
U_+(t)+U_-(t)+U_0\,,
\end{equation}
where $U_0$ is the time-independent operator
\begin{equation}
\label{auxiliary operator U0}
U_0:=P_0+P_{\mathcal{H}^1}\,,
\end{equation}
with $P_0$ and $P_{\mathcal{H}^1}$ 
being orthogonal projections from $\Omega^1$ onto the spaces of
exact 1-forms $\dr\Omega^0$
and
harmonic 1-forms $\mathcal{H}^1$
respectively.
The advantage of working with the operator $V(t)$ is that it is unitary in the full Hilbert space $\Omega^1$ of square-integrable 1-forms.

Observe that the operator $V(t)$ formally satisfies the equation 
\begin{equation}
\label{Cauchy 1}
\left(
-i\,\frac{\partial}{\partial t}+\curl
\right)
V=0
\end{equation}
subject to the initial condition
\begin{equation}
\label{Cauchy 2}
V(0)=\operatorname{Id}\,,
\end{equation}
where $\operatorname{Id}$ is the identity operator in $\Omega^1$.

The essence of our proof of the above theorems is the explicit construction of the operator $V(t)$ as a sum of three oscillatory integrals, modulo the addition of an operator with infinitely smooth (in spatial variables as well as in time) integral kernel.
Here `explicit' refers to the fact that our construction reduces to solving ordinary differential equations.
We will construct the three oscillatory integrals by solving
\eqref{Cauchy 1},
\eqref{Cauchy 2}
modulo the addition of operators with infinitely smooth integral kernels.
This fact will be indicated by using the symbol $\!\!\!\!\mod\Psi^{-\infty}$.
Two of the integrals will be time-dependent (Fourier integral operators) and one time-independent (pseudodifferential operator).

The fact that $\curl$ is not elliptic implies that the Cauchy problem
\eqref{Cauchy 1},
\eqref{Cauchy 2}
is ill-posed, so certain care is required when dealing with it. The relevant justification will be provided in Section~\ref{Justification of the algorithm}. Justification is based on the following two key observations.
\begin{itemize}
\item
The results \cite[Theorem 4.1]{part2} and \cite[Corollary 4.2]{part2} do not require the operator to be elliptic. Namely, the proofs remain valid when one of the eigenvalues of the principal symbol is identically zero.
\item
The operator \emph{extended curl} \cite[Definition~B.1]{curl} acting in $\Omega^1\oplus\Omega^0$ is elliptic and $\delta\Omega^2\oplus\{0\}$ is an invariant subspace of this operator.
\end{itemize}

Let us emphasise that the construction of $V(t)$ is non-standard and somewhat delicate, in that we are dealing with \emph{systems}, as opposed to scalar equations, acting on 1-forms, which give rise, amongst other things, to overdetermined transport equations.

\

Finally, us point out that an alternative way of deriving Weyl coefficients
is by applying methods akin to those in \cite{branson} to the
operator extended curl, which is of Dirac type.
This matter will be explored in a separate paper.

\section{The algorithm}
\label{The algorithm}

In this section we present the algorithm for the explicit construction of the auxiliary operator \eqref{auxiliary operator} modulo $\Psi^{-\infty}$.

\paragraph*{Step 1.}
The principal symbol of $\curl$ reads
\begin{equation*}
\label{principal symbol curl}
[\operatorname{curl}_\mathrm{prin}]_\alpha{}^\beta(x,\xi)= -i\,E_\alpha{}^{\beta\gamma}(x)\,\xi_\gamma\,,
\end{equation*}
recall~\eqref{Levi-Civita tensor}.
The eigenvalues of $\operatorname{curl}_\mathrm{prin}$ are simple and read
\begin{equation}
\label{eigenvalues principal symbol curl}
h^{(0)}(x,\xi)=0\,,\qquad h^{(\pm)}(x,\xi)=\pm\|\xi\|\,,
\end{equation}
for all $(x,\xi)\in T^*M\setminus\{0\}$.
Here and further on we use the notation
\begin{equation}
\label{norm of xi}
\|\xi\|
:=\sqrt{g^{\mu\nu}(x)\,\xi_\mu\xi_\nu}\,.
\end{equation}
Note the dependence of \eqref{norm of xi} on $x$.
By $|\cdot|$ we will denote the Euclidean norm of vectors (no dependence on $x$ here).

We view the $h^{(\pm)}$ as Hamiltonians on the cotangent bundle. The corresponding Hamiltonian flows
$(x^{(\pm)}(t;y,\eta),\xi^{(\pm)}(t;y,\eta))$ are solutions to Hamilton's equations
\begin{equation}
\label{Hamilton's equations}
\dot x^{(\pm)}=h^{(\pm)}_\xi(x^{(\pm)},\xi^{(\pm)})\,,
\qquad
\dot\xi^{(\pm)}=-h^{(\pm)}_x(x^{(\pm)},\xi^{(\pm)})\,,
\end{equation}
with initial conditions
$(x^{(\pm)}(0;y,\eta),\xi^{(\pm)}(0;y,\eta))=(y,\eta)\in T^*M\setminus\{0\}$.
Of course,
$x^{(+)}(t;y,\eta)$ is the geodesic emanating from the point $y\in M$ in the direction
$\,\|\eta\|^{-1}\,g^{\alpha\beta}(y)\,\eta_\beta\,$,
whereas
$x^{(-)}(t;y,\eta)$ is the geodesic emanating from the point $y\in M$ in the opposite direction.

\paragraph*{Step 2.}
Choose phase functions
$\varphi^{(\pm)}(t,x;y,\eta)$ positively homogeneous in $\eta$ of degree one,
$\varphi^{(\pm)}\in C^\infty(\mathbb{R}\times M\times(T^*M\setminus\{0\});\mathbb{C})$,
satisfying the following conditions:
\begin{align}
\label{phase functions condition 1}
\left.\varphi^{(\pm)}\right|_{x=x^{(\pm)}}
&=
0\,,
\\
\label{phase functions condition 2}
\left.\varphi^{(\pm)}_{x^\alpha}\right|_{x=x^{(\pm)}}
&=
\xi^{(\pm)}_\alpha\,,
\\
\label{phase functions condition 3}
\left.\det\varphi^{(\pm)}_{x^\alpha\eta_\beta}\right|_{x=x^{(\pm)}}
&\ne
0\,,
\\
\label{phase functions condition 4}
\operatorname{Im}\varphi^{(\pm)}
&\ge
0\,.
\end{align}

The above definition of phase functions $\varphi^{(\pm)}$ warrants the following remarks.
\begin{itemize}
\item
It is well known that such phase functions exist, see, e.g., \cite{LSV,SaVa,CDV,wave,dirac}.
Furthermore, their construction is straightforward: start with
\eqref{phase functions condition 1}
and
\eqref{phase functions condition 2}
and add a Taylor expansion in powers of $x-x^{(\pm)}$, beginning with quadratic terms.
This Taylor expansion will, of course, depend on the choice of local coordinates $x$.
\item
Condition
\eqref{phase functions condition 3}
is invariant under changes of local coordinates $x$ and $y$.
\item
The quantity
$$\left.\varphi^{(\pm)}_{x^\alpha x^\beta}\right|_{x=x^{(\pm)}}$$
is not a tensor, i.e.~it does not transform in the appropriate way under changes of local coordinates $x$.
However, the quantity
\begin{equation}
\label{imaginary part Hessian phi}
	\left.\operatorname{Im}\varphi^{(\pm)}_{x^\alpha x^\beta}\right|_{x=x^{(\pm)}}
\end{equation}
is a tensor. It is known \cite[Corollary 2.4.5]{SaVa} that if the real symmetric matrix-functions
\eqref{imaginary part Hessian phi}
are strictly positive,
\begin{equation}
\label{phase functions condition 5}
\left.\operatorname{Im}\varphi^{(\pm)}_{x^\alpha x^\beta}\right|_{x=x^{(\pm)}}
>
0\,,
\end{equation}
then condition
\eqref{phase functions condition 3}
is automatically satisfied.
This shows that one can construct complex-valued phase functions globally in time $t\in\mathbb{R}$,
circumventing topological obstructions associated with caustics.
\item
If we use the same local coordinates for $x$ and $y$, then
\begin{equation*}
\label{phase functions fact}
\left.\det\varphi^{(\pm)}_{x^\alpha\eta_\beta}\right|_{t=0,\,x=y}
=
1\,,
\end{equation*}
which implies that if we are interested in constructing the propagators
\eqref{propagator curl plus}
and
\eqref{propagator curl minus}
only for small $t$,
then condition \eqref{phase functions condition 3}
will be automatically satisfied.
In this case there is no need to use complex-valued phase functions.
The more familiar real-valued ones will do.
\end{itemize}

It is convenient for us to choose phase functions $\varphi^{(\pm)}$ so that
$\varphi^{(+)}(0,x;y,\eta)
=
\varphi^{(-)}(0,x;y,\eta)$
and to denote
\begin{equation}
\label{two phase functions coincide at time zero}
\varphi^{(0)}(x;y,\eta)
=
\varphi^{(+)}(0,x;y,\eta)
=
\varphi^{(-)}(0,x;y,\eta)\,.
\end{equation}
Compare with the more restrictive choice of phase functions \cite[formula~(5.5)]{dirac}.

\paragraph*{Step 3.}
We seek the auxiliary operator \eqref{auxiliary operator} in the form
\begin{equation}
\label{definition of V(t)}
V(t)
=
V_+(t)+V_-(t)+V_0
\mod\Psi^{-\infty}\,,
\end{equation}
where the $V_\aleph$, $\aleph\in\{+,-,0\}$, are oscillatory integrals
\begin{align}
\label{time-dependent oscillatory integrals}
V_\pm(t)
:
w_\alpha(x)
&\mapsto
(2\pi)^{-3}
\int
e^{i\varphi^{(\pm)}(t,x;y,\eta)}\,
[v^{(\pm)}]_\alpha{}^\beta(t;y,\eta)\,
\chi^{(\pm)}(t,x;y,\eta)\,
w_\beta(y)\,
\dr y\,\dr\eta\,,
\\
\label{time-independent oscillatory integral}
V_0
:
w_\alpha(x)
&\mapsto
(2\pi)^{-3}
\int
e^{i\varphi^{(0)}(x;y,\eta)}\,
[v^{(0)}]_\alpha{}^\beta(y,\eta)\,
\chi^{(0)}(x;y,\eta)\,
w_\beta(y)\,
\dr y\,\dr\eta\,,
\end{align}
whose symbols $v^{(\aleph)}$, $\aleph\in\{+,-,0\}$, admit asymptotic expansions into components positively homogeneous in $\eta$
\begin{equation}
\label{symbols of propagators U aleph}
[v^{(\aleph)}]_\alpha{}^\beta
\sim
[v^{(\aleph)}_0]_\alpha{}^\beta
+
[v^{(\aleph)}_{-1}]_\alpha{}^\beta
+
\dots,
\end{equation}
where the subscript indicates degree of homogeneity.
The $\chi^{(\aleph)}$, $\aleph\in\{+,-,0\}$, in formulae
\eqref{time-dependent oscillatory integrals}
and
\eqref{time-independent oscillatory integral}
are cut-offs around the singularities, see \cite[Theorem~3.3]{dirac}, and
$\dr y=\dr y^1\dr y^2\dr y^3$,
$\dr\eta=\dr\eta_1\dr\eta_2\dr\eta_3\,$.

Note that in writing our oscillatory integrals in the form
\eqref{time-dependent oscillatory integrals}
and
\eqref{time-independent oscillatory integral}
we use
\emph{right} symbols (no dependence on $x$)
rather than the more familiar
\emph{left} symbols (no dependence on $y$).
Using right symbols is more convenient when dealing with time-dependent operators such as \eqref{auxiliary operator}.

\

\begin{remark}
\label{Let us highlight the difference}
Let us highlight the difference between
the $U_\aleph$, $\aleph\in\{+,-,0\}$, appearing in the RHS of \eqref{auxiliary operator}
and
the $V_\aleph$, $\aleph\in\{+,-,0\}$, appearing in the RHS of \eqref{definition of V(t)}.
The $U_\aleph$ are operators defined by explicit formulae
\eqref{propagator curl plus},
\eqref{propagator curl minus}
and
\eqref{auxiliary operator U0},
whereas the $V_\aleph$ are oscillatory integrals that will be constructed in this section
by means of a formal algorithm.
Of course, we are aiming to show that
\begin{equation}
\label{somewhat delicate matter}
U_\aleph=V_\aleph
\mod\Psi^{-\infty}\,,
\quad
\aleph\in\{+,-,0\}\,.
\end{equation}
However, justification of \eqref{somewhat delicate matter} is a somewhat delicate matter because the operator $\curl$ is not elliptic.
Justification of \eqref{somewhat delicate matter}
will be provided in the next section.
\end{remark}

\

\begin{remark}
\label{remark weights}
For the benefit of the reader, let us mention that the oscillatory integrals appearing in the publication \cite{dirac} cited above have a different structure from \eqref{time-dependent oscillatory integrals}
and
\eqref{time-independent oscillatory integral}. The dissimilarities arise in view of the fact the construction in \cite{dirac} is \emph{global} in space and time, and oscillatory integrals therein feature appropriate weights in the amplitude depending on the phase function. Not only achieving globality in space is not required for the determination of higher order Weyl coefficients, but such an attempt would be frustrated further down the line by the need to choose (normal) local coordinates to carry out explicit calculations. We should also point out that, \emph{a priori}, it is not entirely clear how to achieve a global construction for operators acting on 1-forms.
\end{remark}

\paragraph*{Step 4.}
Write the identity operator on 1-forms as an oscillatory integral
\begin{equation*}
\label{identity operator on 1-forms as an oscillatory integral}
\operatorname{Id}
:
w_\alpha(x)
\mapsto
(2\pi)^{-3}
\int
e^{i\varphi^{(0)}(x;y,\eta)}\,
\operatorname{id}_\alpha{}^\beta(y,\eta)\,
\chi^{(0)}(x;y,\eta)\,
w_\beta(y)\,
\dr y\,\dr\eta
\mod\Psi^{-\infty}\,,
\end{equation*}
where $\varphi^{(0)}$ is the time-independent phase function from Step 2,
see \eqref{two phase functions coincide at time zero}.
The symbol $\operatorname{id}$ admits an asymptotic expansions into components positively homogeneous in $\eta$
\begin{equation}
\label{symbol of identity operator}
\operatorname{id}_\alpha{}^\beta
\sim
[\operatorname{id}_0]_\alpha{}^\beta
+
[\operatorname{id}_{-1}]_\alpha{}^\beta
+
\dots,
\end{equation}
where the subscript indicates degree of homogeneity.

The leading term in \eqref{symbol of identity operator} is the principal symbol of the identity operator and it reads
\begin{equation}
\label{principal symbol of identity operator}
[\operatorname{id}_0]_\alpha{}^\beta(y,\eta)=\delta_\alpha{}^\beta\,,
\end{equation}
where $\delta_\alpha{}^\beta$ is the Kronecker symbol.
The principal symbol $\operatorname{id}_0$ does not depend on the choice of phase function.
However, lower order terms $\operatorname{id}_{-j}$, $j=1,2,\dots$, do depend on the choice of phase function and one needs to determine them
prior to determining the symbols $v^{(\aleph)}$, $\aleph=\{+,-,0\}$ appearing in
\eqref{time-dependent oscillatory integrals}
and
\eqref{time-independent oscillatory integral}.
This matter was addressed in
\cite[Section~6]{wave} for the scalar case.
We do not discuss this issue in the current paper because in all subsequent calculations we will have
$\,\varphi^{(0)}(x;y,\eta)=(x-y)^\alpha\eta_\alpha\,$, with $x$ and $y$ `living' in the same coordinate chart,
in which case $\operatorname{id}_{-j}=0$, $j=1,2,\dots$.

\paragraph*{Step 5.}
Throughout this paper we denote by $P_\pm$ the positive ($+$) and negative ($-$) spectral projections of $\curl$, see \cite[Definition~1.2]{curl},
and by $P_0$ and $P_{\mathcal{H}^1}$ the orthogonal projections from $\Omega^1$ onto the spaces of
exact 1-forms $\dr\Omega^0$
and
harmonic 1-forms $\mathcal{H}^1$
respectively.

Let us denote by $[P^{(\aleph)}]_\alpha{}^\beta(x,\xi)$, $\aleph\in\{+,-,0\}$, the principal symbols of the projection operators $P_\aleph$.
Recall that according to \cite[formula~(3.14)]{curl} we have
\begin{equation}
\label{projection hermiticity}
\overline
{
g_{\alpha\beta'}(x)\,[P^{(\aleph)}]_{\alpha'}{}^{\beta'}(x,\xi)\,g^{\alpha'\beta}(x)
}
=
[P^{(\aleph)}]_\alpha{}^\beta(x,\xi)\,,
\end{equation}
\begin{equation}
\label{projection idempotency}
[P^{(\aleph)}]_\alpha{}^\beta(x,\xi)
\,
[P^{(\beth)}]_\beta{}^\gamma(x,\xi)
=
\delta^{\aleph\beth}
\,
[P^{(\aleph)}]_\alpha{}^\gamma(x,\xi)\,,
\end{equation}
\begin{equation}
\label{projection completeness}
[P^{(+)}]_\alpha{}^\beta(x,\xi)
+
[P^{(-)}]_\alpha{}^\beta(x,\xi)
+
[P^{(0)}]_\alpha{}^\beta(x,\xi)
=
\delta_\alpha{}^\beta\,,
\end{equation}
which means that the $P^{(\aleph)}$ form an orthonormal basis of rank 1 projections in the cotangent fibre.

Put
\begin{equation}
\label{principal symbols of projection operators as functions of time}
P^{(\beth;\pm)}(t;y,\eta):=P^{(\beth)}(x^{(\pm)}(t;y,\eta),\xi^{(\pm)}(t;y,\eta))\,,
\qquad
\beth\in\{+,-,0\}\,.
\end{equation}
Of course, we have
\begin{equation*}
\label{principal symbols of projection operators as functions of time of course}
P^{(\beth;+)}(0;y,\eta)
=
P^{(\beth;-)}(0;y,\eta)
=
P^{(\beth)}(y,\eta)\,,
\qquad
\beth\in\{+,-,0\}\,.
\end{equation*}
It will be convenient for us to use the notation 
\begin{equation}
\label{made this a displayed formula on 20 April 2026}
P^{(\beth;0)}(y,\eta):=P^{(\beth)}(y,\eta)\,.
\end{equation}

Put
\begin{equation}
\label{projected symbols plus minus}
[v^{(\pm;\beth)}_{-j}]_\alpha{}^\beta
:=
[P^{(\beth;\pm)}]_\alpha{}^\gamma
\,
[v^{(\pm)}_{-j}]_\gamma{}^\beta\,,
\qquad
\beth\in\{+,-,0\}\,,
\qquad
j=0,1,\dots,
\end{equation}
where, for the sake of clarity, we suppressed dependence on $(t;y,\eta)$.
Similarly, put
\begin{equation}
\label{projected symbols zero}
[v^{(0;\beth)}_{-j}]_\alpha{}^\beta
:=
[P^{(\beth)}]_\alpha{}^\gamma
\,
[v^{(0)}_{-j}]_\gamma{}^\beta\,,
\qquad
\beth\in\{+,-,0\}\,,
\qquad
j=0,1,\dots,
\end{equation}
\begin{equation}
\label{projected symbols identity}
[\operatorname{id}^{(\beth)}_{-j}]_\alpha{}^\beta
:=
[P^{(\beth)}]_\alpha{}^\gamma
\,
[\operatorname{id}_{-j}]_\gamma{}^\beta\,,
\qquad
\beth\in\{+,-,0\}\,,
\qquad
j=0,1,\dots,
\end{equation}
where, for the sake of clarity, we suppressed dependence on $(y,\eta)$.

In our algorithm the objects
$\operatorname{id}^{(\beth)}_{-j}$
appearing in
\eqref{projected symbols identity}
are the known quantities,
whereas the objects
$v^{(\aleph;\beth)}_{-j}$
appearing in
\eqref{projected symbols plus minus}
and
\eqref{projected symbols zero}
are the unknown quantities (to be determined).
All are rank 1 matrix-functions positively homogeneous in $\eta$ of degree $-j$.
For each $j$ we have nine unknowns
$v^{(\aleph;\beth)}_{-j}$,
$\aleph,\beth\in\{+,-,0\}$.

\paragraph*{Step 6.}
Substitute
\eqref{definition of V(t)}
into the LHS of
\eqref{Cauchy 1}.
We get three oscillatory integrals whose amplitudes pick up the additional dependence on the variable $x$.
Apply the amplitude-to-symbol operator, see Appendix~\ref{The amplitude-to-symbol operator},
to each of the three oscillatory integrals.
This gives us three oscillatory integrals with symbols
$w^{(\pm)}(t;y,\eta)$ and $w^{(0)}(y,\eta)$, respectively.
These admit asymptotic expansions into components positively homogeneous in $\eta$
\begin{equation}
\label{asymptotic expansiions for the w}
[w^{(\aleph)}]_\alpha{}^\beta
\sim
[w^{(\aleph)}_1]_\alpha{}^\beta
+
[w^{(\aleph)}_0]_\alpha{}^\beta
+
[w^{(\aleph)}_{-1}]_\alpha{}^\beta
+
\dots.
\end{equation}
Note that now the leading degree of homogeneity is $+1$.

Introduce the rank 1 matrix-functions
\begin{equation}
\label{Introduce the rank 1 matrix-functions}
[w^{(\aleph;\beth)}_{-j}]_\alpha{}^\beta
:=
[P^{(\beth;\aleph)}]_\alpha{}^\gamma
\,
[w^{(\aleph)}_{-j}]_\gamma{}^\beta\,,
\qquad
\aleph,\beth\in\{+,-,0\}\,,
\qquad
j=-1,0,1,\dots,
\end{equation}
and consider the hierarchy of linear equations
\begin{equation}
\label{hierarchy of linear equations}
w^{(\aleph;\beth)}_{-j}=0,
\qquad
\aleph,\beth\in\{+,-,0\}\,,
\qquad
j=-1,0,1,\dots,
\end{equation}
for the unknowns
\begin{equation*}
\label{the unknowns in the hierarchy of linear equations}
v^{(\aleph;\beth)}_{-j},
\qquad
\aleph,\beth\in\{+,-,0\},
\qquad
j=0,1,\dots.
\end{equation*}
Some of the linear operators appearing in the hierarchy
\eqref{hierarchy of linear equations}
are algebraic (multiplication by a given matrix-function)
whereas others are first order ordinary differential operators in $t$.

Henceforth, we will refer to \eqref{hierarchy of linear equations} as \emph{transport equations}.

The remainder of the algorithm is concerned with solving \eqref{hierarchy of linear equations}
recursively subject to initial conditions
\begin{equation}
\label{initial conditions for the unknowns in the hierarchy of linear equations}
\left.
v^{(+;\beth)}_{-j}
\right|_{t=0}
+
\left.
v^{(-;\beth)}_{-j}
\right|_{t=0}
+
v^{(0;\beth)}_{-j}
=
\operatorname{id}^{(\beth)}_{-j}\,,
\qquad
\beth\in\{+,-,0\},
\qquad
j=0,1,\dots.
\end{equation}

\paragraph*{Step 7.}
Consider \eqref{hierarchy of linear equations} for $j=-1$, which is purely algebraic.
We call this the zeroth transport equation.
Examination of the zeroth transport equation gives us
\begin{equation}
\label{v0 for aleph different from beth}
v^{(\aleph;\beth)}_0=0
\quad\text{for}\quad
\aleph\ne\beth\,.
\end{equation}

The zeroth transport equation does not allow one to determine the $v^{(\aleph;\aleph)}_0$. These will be determined at the next step.

\paragraph*{Step 8a.}
Consider \eqref{hierarchy of linear equations} for $j=0$.
We call this the first transport equation.
We examine it subject to initial conditions
\eqref{initial conditions for the unknowns in the hierarchy of linear equations}, $j=0$.

Consider first the equation $w^{(+;+)}_0=0$. This will allow us to determine $v^{(+;+)}_0$. 

Let $[b^{(+)}]_\beta(x,\xi)$ be an eigencovector of $[\curl_\prin]_\alpha{}^\beta(x,\xi)$ corresponding to the eigenvalue $+\|\xi\|$, normalised in terms of the metric at the point $x$.

\

\begin{remark}
Introducing the eigencovector of the principal symbol of the operator $\curl$
one encounters two issues which, at first glance, appear to be problematic.
\begin{itemize}
\item
Unlike the projection onto an eigenspace of the principal symbol of the operator $\curl$,
the eigencovector of the principal symbol of the operator $\curl$ is not uniquely defined.
It is defined modulo a gauge transformation,
namely,
\begin{equation}
\label{version 26 equation 1}
b^{(+)}
\mapsto
e^{i\phi^{(+)}}\,b^{(+)}\,,
\qquad
\phi^{(+)}
=
\phi^{(+)}(x,\xi)\,,
\end{equation}
where $\phi^{(+)}(x,\xi)$ is an arbitrary real-valued scalar function.
\item
There are topological obstructions to a smooth choice of gauge,
even locally at a fixed point $x\in M$.
See \cite[Proposition~3.2]{obstructions} for details.
\end{itemize}
However, these issues do not affect the final formula for
the invariantly defined matrix-function
$[v^{(+;+)}_0]_\alpha{}^\beta(t;y,\eta)$.
\end{remark}

\

We have
\begin{equation}
\label{decompostion of v plus plus zero}
[v^{(+;+)}_0]_\alpha{}^\beta(t;y,\eta)
=
[b^{(+)}]_\alpha(x^{(+)}(t;y,\eta),\xi^{(+)}(t;y,\eta))
\,
[f^{(+)}_0]^\beta(t;y,\eta)\,,
\end{equation}
where the $[f^{(+)}_0]^\beta(t;y,\eta)$ are some functions positively homogeneous in $\eta$ of degree 0.
Equation $w^{(+;+)}_0=0$ reduces to the ordinary differential equation
\begin{equation}
\label{first transport equation ODE plus}
-i
\frac{\dr[f^{(+)}_0]^\beta}{\dr t}
+
q^{(+)}_0\,[f^{(+)}_0]^\beta
=0\,,
\end{equation}
where $q^{(+)}_0(t;y,\eta)$ is a function positively homogeneous in $\eta$ of degree 0 given by the explicit formula
\begin{multline}
\label{formula for q 4}
q^{(+)}_0
=
-i
\left[
\,
\overline
{	
[b^{(+)}]_\alpha
}
\,g^{\alpha\beta}
\left(
\frac{\partial [b^{(+)}]_\beta}{\partial x^\gamma}
\,
\frac{\partial  h^{(+)}}{\partial\xi_\gamma}
\,-\,
\frac{\partial [b^{(+)}]_\beta}{\partial\xi_\gamma}
\,
\frac{\partial  h^{(+)}}{\partial x^\gamma}
\right)
\right]_{(x,\xi)=(x^{(+)},\,\xi^{(+)})}
\\
+
i
\left[
\,
\overline
{	
[b^{(+)}]_\kappa
}
\,g^{\kappa\lambda}
\right]_{(x,\xi)=(x^{(+)},\,\xi^{(+)})}
\times
\\
\times
\left[
\left(
\frac{\partial}{\partial\eta_\beta}
L^{(+)}_\beta
-
\frac{1}{2}
\varphi^{(+)}_{\eta_\alpha\eta_\beta}
L^{(+)}_\alpha
L^{(+)}_\beta
\right)
\left(
	\varphi^{(+)}_t\,\delta_\lambda{}^\mu
	-i
	E_\lambda{}^{\mu\nu}(x)\,
	\varphi^{(+)}_{x^\nu}
	\right)
	[b^{(+)}]_\mu(x^{(+)},\xi^{(+)})
	\right]_{x=x^{(+)}}.
\end{multline}
See Appendix~\ref{The amplitude-to-symbol operator} for the definition of the differential operators $L^{(+)}_\alpha\,$.

Note that under the gauge transformation \eqref{version 26 equation 1} the function $q^{(+)}_0$ transforms as
\begin{equation}
\label{version 26 equation 2}
q^{(+)}_0
\mapsto
q^{(+)}_0
+
\left(
\frac{\partial\phi^{(+)}}{\partial x^\gamma}
\,
\frac{\partial  h^{(+)}}{\partial\xi_\gamma}
\,-\,
\frac{\partial\phi^{(+)}}{\partial\xi_\gamma}
\,
\frac{\partial  h^{(+)}}{\partial x^\gamma}
\right)_{(x,\xi)=(x^{(+)},\,\xi^{(+)})}
=
q^{(+)}_0
+
\frac{\dr\phi^{(+)}(x^{(+)},\,\xi^{(+)})}{\dr t}\,.
\end{equation}

Formula
\eqref{initial conditions for the unknowns in the hierarchy of linear equations}
with $\beth=+$
and
formulae
\eqref{v0 for aleph different from beth},
\eqref{projected symbols identity},
\eqref{principal symbol of identity operator}
imply
\begin{equation*}
\label{v++0}
[v^{(+;+)}_0]_\alpha{}^\beta(0;y,\eta)
=
[P^{(+)}]_\alpha{}^\beta(y,\eta)\,,
\end{equation*}
so that the initial condition for $[f^{(+)}_0]^\beta$ reads
\begin{equation}
\label{initial condition for f plus}
[f^{(+)}_0]^\beta(0;y,\eta)
=
g^{\beta\gamma}(y)\,
\overline
{
[b^{(+)}]_\gamma(y,\eta)
}\,.
\end{equation}

The solution to the Cauchy problem
\eqref{first transport equation ODE plus},
\eqref{initial condition for f plus}
reads
\begin{equation}
\label{explicit formula for f plus}
[f^{(+)}_0]^\beta(t;y,\eta)
=
e^{-i\int_0^t q^{(+)}_0(\tau;y,\eta)\,\dr\tau}\,
g^{\beta\gamma}(y)\,
\overline
{
[b^{(+)}]_\gamma(y,\eta)
}\,.
\end{equation}
Substituting
\eqref{explicit formula for f plus}
into
\eqref{decompostion of v plus plus zero}
we obtain the formula for $[v^{(+;+)}_0]_\alpha{}^\beta(t;y,\eta)$.

Examination of formulae
\eqref{version 26 equation 1},
\eqref{decompostion of v plus plus zero},
\eqref{version 26 equation 2}
and
\eqref{explicit formula for f plus}
shows that the formula for $v^{(+;+)}_0$
is invariant under gauge transformations \eqref{version 26 equation 1} of
the eigencovector of the principal symbol of the operator $\curl$.

The formula for $[v^{(-;-)}_0]_\alpha{}^\beta(t;y,\eta)$ is obtained by solving the equation $w^{(-;-)}_0=0$ in a similar fashion.

\paragraph*{Step 8b.}
Formula
\eqref{initial conditions for the unknowns in the hierarchy of linear equations}
with $\beth=0$
and
formulae
\eqref{v0 for aleph different from beth},
\eqref{projected symbols identity},
\eqref{principal symbol of identity operator}
imply
\begin{equation*}
\label{v000}
[v^{(0;0)}_0]_\alpha{}^\beta(y,\eta)
=
[P^{(0)}]_\alpha{}^\beta(y,\eta)\,.
\end{equation*}

Let us now focus our attention on equations $w^{(0;\beth)}_0=0$, $\beth\in\{+,-,0\}$.
We are looking at a system of three time-independent linear algebraic equations
for two time-independent unknowns $v^{(0;\pm)}_{-1}\,$,
i.e.~we are looking at an overdetermined system.
It seems that we have encountered an impasse.

We overcome this overdeterminacy by choosing to \emph{disregard} equation $w^{(0;0)}_0=0$. 
It will be shown in
Section~\ref{Justification of the algorithm}
that this equation is satisfied \emph{automatically}.

\paragraph*{Step 8c.}
Equations $w^{(\aleph;\beth)}_0=0$, $\aleph\ne\beth$,
are algebraic equations for the unknowns
$v^{(\aleph;\beth)}_{-1}$, $\aleph\ne\beth$.
These algebraic equations are of the form
$\,(h^{(\aleph)}-h^{(\beth)})v^{(\aleph;\beth)}_{-1}=\dots\,$,
where the right-hand sides $\,\dots\,$ are known,
so the algebraic equations in questions can be solved explicitly.

Thus, we have determined the
$v^{(\aleph;\aleph)}_0$
and the
$v^{(\aleph;\beth)}_{-1}$, $\aleph\ne\beth$.
This concludes our analysis of the first transport equation.

\paragraph*{Step 9.}
Suppose that we have analysed the first $k$ transport equations.
This means that we have solved
\eqref{hierarchy of linear equations}
for $j=-1,0,1,2,\dots,k-1$
and determined the
$v^{(\aleph;\aleph)}_{-j}$ for $j=1,2,\dots,k-1$
and the
$v^{(\aleph;\beth)}_{-j}$, $\aleph\ne\beth$, for $j=1,2,\dots,k$.
We now need to solve
\eqref{hierarchy of linear equations}
for $j=k$
and determine the
$v^{(\aleph;\aleph)}_{-k}$
and the
$v^{(\aleph;\beth)}_{-k-1}$, $\aleph\ne\beth$.

Then the argument from Step 8 can be repeated with minor adjustments. Namely,
formulae
\eqref{decompostion of v plus plus zero}
and
\eqref{first transport equation ODE plus}
now read
\begin{equation}
\label{decompostion of v plus plus zero adjustment}
[v^{(+;+)}_{-k}]_\alpha{}^\beta(t;y,\eta)
=
[b^{(+)}]_\alpha(x^{(+)}(t;y,\eta),\xi^{(+)}(t;y,\eta))
\,
[f^{(+)}_{-k}]^\beta(t;y,\eta)\,,
\end{equation}
\begin{equation}
\label{first transport equation ODE plus adjustment}
-i
\frac{\dr[f^{(+)}_{-k}]^\beta}{\dr t}
+
q^{(+)}_0\,[f^{(+)}_{-k}]^\beta
+
[r^{(+)}_{-k}]^\beta
=0\,,
\end{equation}
where the $[f^{(+)}_{-k}]^\beta(t;y,\eta)$ are some unknown functions positively homogeneous in $\eta$ of degree $-k$ and the $r^{(+)}_{-k}(t;y,\eta)$ are some given functions positively homogeneous in $\eta$ of degree $-k$.
The $q^{(+)}_0(t;y,\eta)$ in
\eqref{first transport equation ODE plus adjustment}
is the same as in
\eqref{first transport equation ODE plus}:
it is positively homogeneous in $\eta$ of degree 0 and given by the explicit formula
\eqref{formula for q 4}.

Equation
\eqref{first transport equation ODE plus adjustment}
is accompanied by an initial condition originating from
\eqref{initial conditions for the unknowns in the hierarchy of linear equations}.

As in Step 8b, we choose to \emph{disregard} equation $w^{(0;0)}_{-k}=0$. 
It will be shown in
Section~\ref{Justification of the algorithm}
that this equation is satisfied \emph{automatically}.

\

\begin{remark}
\label{gauge transformations now sit at the source}
Our construction of the $v^{(+;+)}_{-k}$, $k=0,1,2,\dots$,
is invariant under gauge transformations \eqref{version 26 equation 1} of
the eigencovector of the principal symbol of the operator $\curl$,
however in practical calculations one has to choose a gauge.
It is convenient to define $[b^{(+)}_\shortparallel]_\alpha(t;y,\eta)$ as the
result of parallel transport of the covector $[b^{(+)}]_\alpha(y,\eta)$
along the geodesic $x^{(+)}(\tau;y,\eta)$
from $\tau=0$ to $\tau=t$.
Then $[b^{(+)}_\shortparallel]_\alpha(t;y,\eta)$ is an
eigencovector of $[\curl_\prin]_\alpha{}^\beta(x^{(+)}(t;y,\eta),\xi^{(+)}(t;y,\eta))$
and we can replace
\eqref{decompostion of v plus plus zero}
and
\eqref{decompostion of v plus plus zero adjustment}
with
\begin{equation}
\label{decompostion of v plus plus zero adjustment parallel}
[v^{(+;+)}_{-k}]_\alpha{}^\beta(t;y,\eta)
=
[b^{(+)}_\shortparallel]_\alpha(t;y,\eta)
\,
[f^{(+)}_{-k,\shortparallel}]^\beta(t;y,\eta)\,,
\qquad k=0,1,2,\dots.
\end{equation}
The advantage of the representation
\eqref{decompostion of v plus plus zero adjustment parallel}
is that the gauge transformation
now sits at the source,
\begin{equation}
\label{version 26 equation 1 parallel}
b^{(+)}_\shortparallel
\mapsto
e^{i\phi^{(+)}}\,b^{(+)}_\shortparallel\,,
\qquad
\phi^{(+)}
=
\phi^{(+)}(y,\eta)\,,
\end{equation}
where $\phi^{(+)}(y,\eta)$ is an arbitrary real-valued scalar function.
Compare
\eqref{version 26 equation 1}
and
\eqref{version 26 equation 1 parallel}:
now we have $\phi^{(+)}=\phi^{(+)}(y,\eta)$
rather than $\phi^{(+)}=\phi^{(+)}(x,\xi)$.
This will come handy in Section~\ref{Implementation of the algorithm}
when we will fix the $y$.

Formula \eqref{formula for q 4} now reads
{\small
\begin{multline}
\label{formula for q 4 parallel}
q^{(+)}_{0,\shortparallel}
=
-i
\,
\overline
{	
[b^{(+)}_\shortparallel]_\alpha
}
\,g^{\alpha\beta}(x^{(+)})\,
\frac{\dr [b^{(+)}_\shortparallel]_\beta}{\dr t}
\\
+
i
\,
\overline
{	
[b^{(+)}_\shortparallel]_\kappa
}
\,g^{\kappa\lambda}(x^{(+)})
\left[
\left(
\frac{\partial}{\partial\eta_\beta}
L^{(+)}_\beta
-
\frac{1}{2}
\varphi^{(+)}_{\eta_\alpha\eta_\beta}
L^{(+)}_\alpha
L^{(+)}_\beta
\right)
\left(
	\varphi^{(+)}_t\,\delta_\lambda{}^\mu
	-i
	E_\lambda{}^{\mu\nu}(x)\,
	\varphi^{(+)}_{x^\nu}
	\right)
	[b^{(+)}_\shortparallel]_\mu
	\right]_{x=x^{(+)}}.
\end{multline}
}
It is easy to see that the RHS of \eqref{formula for q 4 parallel}
is invariant under gauge transformations
\eqref{version 26 equation 1 parallel}
and so are the resulting
$v^{(+;+)}_{-k}$, $k=0,1,2,\dots$.
\end{remark}

\section{Justification of the algorithm}
\label{Justification of the algorithm}

In the previous section we presented a formal algorithm for the construction of three oscillatory integrals,
two time-dependent oscillatory integrals $V_\pm(t)$
and one time-independent oscillatory integral $V_0$.
The task at hand is to justify that these describe,
modulo the addition of operators with infinitely smooth (in all variables) integral kernels,
the operators $U_+(t)$,  $U_-(t)$ and $U_0$ defined in accordance with formulae
\eqref{propagator curl plus},
\eqref{propagator curl minus}
and
\eqref{auxiliary operator U0}
respectively.
In this section we prove \eqref{somewhat delicate matter}.

\

\noindent\textbf{\textsc{Part 1.}}

The algorithm described in Section~\ref{The algorithm} requires us to disregard
time-independent transport equations $w^{(0;0)}_{-k}=0$, $k=0,1,2,\dots$,
see Steps 8b and 9.
In this part we show that we automatically get
\begin{equation}
\label{disregard equations}
w^{(0;0)}_{-k}=0,\quad k=0,1,2,\dots.
\end{equation}

Let $W$ be the pseudodifferential operator with
phase function $\varphi^{(0)}(x;y,\eta)$ and right symbol
\begin{equation*}
\label{20 April 2026 equation 1}
[w^{(0;0)}_0]_\alpha{}^\beta(y,\eta)
+
[w^{(0;0)}_{-1}]_\alpha{}^\beta(y,\eta)
+
[w^{(0;0)}_{-2}]_\alpha{}^\beta(y,\eta)
+
\dots,
\end{equation*}
see also
\eqref{asymptotic expansiions for the w},
\eqref{Introduce the rank 1 matrix-functions}
and
\eqref{made this a displayed formula on 20 April 2026}.
We have
\begin{equation}
\label{20 April 2026 equation 2}
\curl V_0=W\mod\Psi^{-\infty}
\end{equation}
because the operator $W$ plays the role of error term in our algorithm.

Suppose that \eqref{disregard equations} is false. Let $l$
be the smallest value of $k$ for which \eqref{disregard equations} fails.
Then the operator $W$ is of order $\,-l\,$ and its principal symbol $W_\prin$
(of degree $-l$) is such that
\begin{equation}
\label{20 April 2026 equation 3}
W_\prin(y,\eta)\ne0
\quad
\text{for some}
\quad
(y,\eta)\in T^*M\setminus\{0\}\,.
\end{equation}

Acting on
\eqref{20 April 2026 equation 2}
with the projection operator $\,P_0\,$,
we get
\begin{equation}
\label{20 April 2026 equation 4}
P_0\,W=0\mod\Psi^{-\infty}.
\end{equation}
The operator $\,P_0\,W\,$ is of order $\,-l\,$ and its principal symbol
$\,(P_0\,W)_\prin\,$ (of degree $-l$) is obtained by multiplying the principal symbols of
$\,P_0\,$ and $W$.
Formulae
\eqref{Introduce the rank 1 matrix-functions}
and
\eqref{made this a displayed formula on 20 April 2026}
imply
\begin{equation}
\label{20 April 2026 equation 5}
(P_0\,W)_\prin=W_\prin\,.
\end{equation}

Formulae
\eqref{20 April 2026 equation 5}
and
\eqref{20 April 2026 equation 3}
contradict
\eqref{20 April 2026 equation 4}.

\

\noindent\textbf{\textsc{Part 2.}}

Our oscillatory integrals $V_\aleph$, $\aleph\in\{+,-,0\}$, satisfy the identities
\begin{equation}
\label{our algorithm implies the identities}
P_\beth\,V_\aleph=0\mod\Psi^{-\infty}
\quad\text{for}\quad
\aleph\ne\beth\,.
\end{equation}
These are established by means of an argument similar to that presented in the proof of \cite[Theorem~4.1]{part2}.

For concreteness, let us show~\eqref{our algorithm implies the identities} for the case $\aleph=+$ and $\beth=-$; the other cases are dealt with analogously. 
Arguing by contradiction, suppose that
\begin{equation}
\label{proof justification matteo 1}
P_-\,V_+(t)\in C^\infty(\mathbb{R};\Psi^{-k}) \quad \text{but} \quad P_-\,V_+(t)\not\in C^\infty(\mathbb{R};\Psi^{-k-1})
\end{equation}
for some $k>0$.
Since $[\curl,P_-]\in \Psi^{-\infty}$, the construction of the symbol of $V_\pm(t)$ from the previous section implies that the operator $P_-V_+(t)$ satisfies
\begin{equation}
\label{proof justification matteo 2}
(-i\partial_t+\curl)(P_-V_+(t))=0 \mod \Psi^{-\infty}\,.
\end{equation}
Examining the leading transport equation arising from \eqref{proof justification matteo 2} at the level of the symbol (e.g., arguing along the lines of \cite[Proof of Theorem~4.1]{part2}) gives the identity\footnote{Here the notation $A_{\mathrm{prin},k}$ 
stands for the leading term, of degree of homogeneity of $-k$, of the symbol of the oscillatory integral $A$.}
\begin{equation}
\label{proof justification matteo 3}
P^{(+;+)}[P_-V_+(t)]_{\mathrm{prin},k}=[P_-V_+(t)]_{\mathrm{prin},k}\,.
\end{equation}
On the other hand, the identity $P_-^2V_+(t)=P_-V_+(t)\mod \Psi^{-\infty}$ implies
\begin{equation}
\label{proof justification matteo 3bis}
P^{(-;+)}[P_-V_+(t)]_{\mathrm{prin},k}=[P_-V_+(t)]_{\mathrm{prin},k}\,.
\end{equation}
Since $P^{(+;+)}P^{(-;+)}=0$, formulae~\eqref{proof justification matteo 3} and~\eqref{proof justification matteo 3bis} in turn 
give us
\begin{equation*}
\label{proof justification matteo 4}
[P_-V_+(t)]_{\mathrm{prin},k}=0\,.
\end{equation*}
The latter contradicts~\eqref{proof justification matteo 1}.

\

\noindent\textbf{\textsc{Part 3.}}

Arguing as in \cite[Corollary~4.2]{part2}, we see that
\eqref{our algorithm implies the identities} implies
\begin{equation}
\label{Arguing as in Corollary 4 point 2 pm}
V_\pm(0)=P_\pm\mod\Psi^{-\infty}\,,
\end{equation}
\begin{equation}
\label{Arguing as in Corollary 4 point 2 0}
V_0=P_0\mod\Psi^{-\infty}\,.
\end{equation}

For the sake of concreteness, let us show~\eqref{Arguing as in Corollary 4 point 2 0}. The proof of~\eqref{Arguing as in Corollary 4 point 2 pm} proceeds analogously. 
By construction, we have
\begin{equation}
\label{proof justification matteo 6}
V_+(0)+V_-(0)+V_0=\mathrm{Id} \mod \Psi^{-\infty}\,.
\end{equation}
Furthermore, \eqref{our algorithm implies the identities} implies
\begin{equation}
\label{proof justification matteo 7}
V_0=P_0 V_0 \mod \Psi^{-\infty}\,.
\end{equation}
Acting on~\eqref{proof justification matteo 6} on the left first with $P_++P_-+P_0$ and then with $P_0$, on account of~\eqref{proof justification matteo 7} and the almost-orthogonality of pseudodifferential projections, we obtain
\begin{equation}
\label{proof justification matteo 8}
P_0 V_0=P_0 \mod \Psi^{-\infty}\,.
\end{equation}
Combining formulae~\eqref{proof justification matteo 8} and~\eqref{proof justification matteo 7} we arrive at~\eqref{Arguing as in Corollary 4 point 2 0}.

Formula
\eqref{Arguing as in Corollary 4 point 2 0}
implies
\eqref{somewhat delicate matter}
with $\aleph=0$.
However, formula
\eqref{Arguing as in Corollary 4 point 2 pm}
does not imply
\eqref{somewhat delicate matter}
with $\aleph=\pm$ because of the presence of the time variable.

\begin{remark}
\label{remark P0 is a rank 1 matrix-function}
Consider the operator $P_0\,$, the orthogonal projection from $\Omega^1$ onto the space of exact 1-forms $\dr\Omega^0,$ and write it as a pseudodifferential operator with right symbol as in
\eqref{time-independent oscillatory integral}
and
\eqref{Arguing as in Corollary 4 point 2 0}.
Suppose we are using the phase function
$\varphi^{(0)}=(x-y)^\gamma\eta_\gamma\,$.
Then the full right symbol $[v^{(0)}]_\alpha{}^\beta(y,\eta)$ of the operator $P_0$ has the remarkable property that it is a rank 1 matrix-function. Namely, we have
\begin{equation*}
\label{remark P0 is a rank 1 matrix-function equation 1}
[v^{(0)}]_\alpha{}^\beta(y,\eta)
=
i\eta_\alpha\,
a^\beta(y,\eta)\,,
\end{equation*}
where $a^\beta(y,\eta)$ is the full right symbol of the pseudodifferential operator
\begin{equation*}
\label{remark P0 is a rank 1 matrix-function equation 2}
{-}\Delta^{-1}\delta:\Omega^1\to\Omega^0\,,
\end{equation*}
which is of order $-1$.
This follows immediately from \cite[formula~(5.5)]{curl}.

The fact that the right symbol of $P_0$ factorises is useful for running checks on calculations.
\end{remark}

\

\noindent\textbf{\textsc{Part 4.}}

Next, we prove
\eqref{somewhat delicate matter}
for $\aleph=\pm$.

According to the algorithm from Section~\ref{The algorithm}, the oscillatory integrals $V_\pm(t)$ satisfy equations
\begin{equation}
\label{Cauchy 1 pm}
\left(
-i\,\frac{\partial}{\partial t}+\curl
\right)
V_\pm=0
\mod\Psi^{-\infty}
\end{equation}
subject to initial conditions \eqref{Arguing as in Corollary 4 point 2 pm}.
We need to show that \eqref{Cauchy 1 pm} and \eqref{Arguing as in Corollary 4 point 2 pm}
imply
\eqref{somewhat delicate matter}
for $\aleph=\pm$.

Consider the oscillatory integrals $(\delta V_\pm)(t)$. These are operators mapping 1-forms to scalar functions. Formula \eqref{Arguing as in Corollary 4 point 2 pm} implies
\begin{equation}
\label{2 April 2026 equation 1}
(\delta V_\pm)(0)=0\mod\Psi^{-\infty}\,.
\end{equation}
Acting with $\delta$ on \eqref{Cauchy 1 pm} we get
\begin{equation}
\label{2 April 2026 equation 2}
\frac{\partial}{\partial t}
\delta V_\pm=0
\mod\Psi^{-\infty}\,.
\end{equation}
Formulae
\eqref{2 April 2026 equation 1}
and
\eqref{2 April 2026 equation 2}
imply
\begin{equation}
\label{2 April 2026 equation 3}
(\delta V_\pm)(t)=0\mod\Psi^{-\infty}\,.
\end{equation}

Consider the operator \emph{extended curl}
\begin{equation}
\label{definition extended curl}
\operatorname{curl}_E:=
\begin{pmatrix}
\operatorname{curl}&\dr\\
\delta&0
\end{pmatrix}:
(\Omega^1  \cap H^1) \oplus (\Omega^0\cap H^1)
\to
\Omega^1\oplus \Omega^0\,,
\end{equation}
see \cite[Definition~B.1]{curl}. It acts on 2-columns where the first entry is a 1-form and the second --- a scalar function. The advantage of working with extended curl is that it is elliptic.

Put
\begin{equation}
\label{2 April 2026 equation 5}
V_{E,\pm}(t)
:=
\begin{pmatrix}
V_\pm(t)&0
\\
0&0
\end{pmatrix}.
\end{equation}
Formulae
\eqref{2 April 2026 equation 3}--\eqref{2 April 2026 equation 5}
and
\eqref{Arguing as in Corollary 4 point 2 pm}
imply
\begin{equation}
\label{2 April 2026 equation 6}
\left(
-i\,\frac{\partial}{\partial t}+\curl_E
\right)
V_{E,\pm}=0
\mod\Psi^{-\infty}\,,
\end{equation}
\begin{equation}
\label{2 April 2026 equation 7}
V_{E,\pm}(0)
=
\begin{pmatrix}
P_\pm&0
\\
0&0
\end{pmatrix}
\mod\Psi^{-\infty}\,.
\end{equation}
Here we are looking at a well-posed Cauchy problem
for the unknown operator $V_{E,\pm}(t)$.

The Cauchy problem
\eqref{2 April 2026 equation 6},
\eqref{2 April 2026 equation 7}
can be solved by writing a series expansion in eigenfunctions of extended curl. The latter were described in \cite[subsection~B.2]{curl}. Straightforward analysis of this series gives
\begin{equation}
\label{2 April 2026 equation 8}
V_{E,\pm}(t)
=
\begin{pmatrix}
U_\pm(t)&0
\\
0&0
\end{pmatrix}
\mod\Psi^{-\infty}\,.
\end{equation}

Formulae
\eqref{2 April 2026 equation 5}
and
\eqref{2 April 2026 equation 8}
imply
\eqref{somewhat delicate matter}
with $\aleph=\pm$.

\begin{remark}
It can be shown that the propagators of $\curl$ and extended curl are related as
\begin{equation*}
\label{propagator of extended curl}
e^{-it\curl_E}
=
\begin{pmatrix}
U(t)+\dr\cos(t\sqrt{-\Delta})(-\Delta)^{-1}\delta+P_{\mathcal{H}^1}
&
-i\dr(-\Delta)^{-1/2}\sin(t\sqrt{-\Delta})
\\
-i\sin(t\sqrt{-\Delta})(-\Delta)^{-1/2}\delta
&
\cos(t\sqrt{-\Delta})
\end{pmatrix}.
\end{equation*}
This follows from \cite[formula~(B.6)]{curl}.
Here $\Delta$ is Laplace--Beltrami operator,
and the negative powers of $\,-\Delta\,$ refer to the pseudoinverse
(projection onto the eigenspace corresponding to the zero eigenvalue being excluded, as in Chapter 2 Section 2 of \cite{rellich}).

It it is easy to see that
\begin{equation*}
\label{2 April 2026 equation 9}
\begin{pmatrix}
P_\pm&0
\\
0&0
\end{pmatrix}
e^{-it\curl_E}
=
e^{-it\curl_E}
\begin{pmatrix}
P_\pm&0
\\
0&0
\end{pmatrix}
=
\begin{pmatrix}
U_\pm(t)&0
\\
0&0
\end{pmatrix}
\end{equation*}
which agrees with \eqref{2 April 2026 equation 8}.
\end{remark}


\

\noindent\textbf{\textsc{Part 5.}}

Finally, we clarify why the equation \(w^{(\pm;\pm)}_{-k} = 0\) does indeed reduce to an ordinary differential equation in time in the unknown \(v^{(\pm;\pm)}_{-k}\), as was claimed in Steps 8a and 9 of the algorithm described in Section \ref{The algorithm}.
This is a consequence of the following key identity:
\begin{equation}
\label{eikonal magic}
\small
	g^{\beta\delta}\bigl(x^{(\pm)}\bigr)\,\overline{b^{(\pm)}_\delta\bigl(x^{(\pm)},\xi^{(\pm)}\bigr)}\left.\frac{\partial}{\partial x^\sigma}\!\left(\frac{\partial\varphi^{(\pm)}}{\partial t}\,\delta_\beta{}^\gamma +[\curl_{\prin}]_\beta{}^\gamma\bigl(x,\mathrm{d}_x\varphi^{(\pm)}\bigr)\right)\right|_{x=x^{(\pm)}}b^{(\pm)}_\gamma\bigl(x^{(\pm)},\xi^{(\pm)}\bigr) = 0\,.
\end{equation}

\begin{remark}
	Formula \eqref{eikonal magic} follows from its scalar analogue \cite[Theorem~2.4.16]{SaVa}, which in our setting reads
	\begin{equation}
	\label{eq:scalar_eikonal}
	\left.\frac{\partial}{\partial x^\sigma}\!\left(\frac{\partial\varphi^{(\pm)}}{\partial t}(t,x;y,\eta) +h^{(\pm)}\bigl(x,\mathrm{d}_x\varphi^{(\pm)}(t,x;y,\eta)\bigr)\right)\right|_{x=x^{(\pm)}(t;y,\eta)} = 0\,.
	\end{equation}
	The function in round brackets in \eqref{eq:scalar_eikonal} is known as the (\emph{scalar}) \emph{eikonal function} \(\mathrm{e}^{(\pm)}(t,x;y,\eta)\). Since \(h^{(\pm)}\) is positively homogeneous in momentum of degree 1, formula \eqref{eq:scalar_eikonal} tells us that the scalar eikonal function has a second order zero in \(x\) at \(x=x^{(\pm)}(t;y,\eta)\).
\end{remark}

\section{Implementation of the algorithm}
\label{Implementation of the algorithm}

\subsection{Simplifications}
\label{Simplifications}

Implementation of the algorithm described in Section~\ref{The algorithm} is a challenging task. In this section we describe a number of simplifications that make the algorithm more manageable.

\paragraph*{Simplification 1.}
Drop the $\,w_\beta(y)\,\dr y\,$ in
\eqref{time-dependent oscillatory integrals}
and
\eqref{time-independent oscillatory integral}
and consider the oscillatory integrals
\begin{align}
\label{time-dependent oscillatory integrals dropped}
&(2\pi)^{-3}
\int
e^{i\varphi^{(\pm)}(t,x;y,\eta)}\,
[v^{(\pm)}]_\alpha{}^\beta(t;y,\eta)\,
\chi^{(\pm)}(t,x;y,\eta)\,
\dr\eta\,,
\\
\label{time-independent oscillatory integral dropped}
&(2\pi)^{-3}
\int
e^{i\varphi^{(0)}(x;y,\eta)}\,
[v^{(0)}]_\alpha{}^\beta(y,\eta)\,
\chi^{(0)}(x;y,\eta)\,
\dr\eta
\end{align}
instead.
The quantity
\eqref{time-dependent oscillatory integrals dropped}
is a distribution in the variables $t$ and $x$,
whereas
\eqref{time-independent oscillatory integral dropped}
is a distribution in the variable $x$.
The $y\in M$ (source) takes on the role of parameter.

\paragraph*{Simplification 2.}
Choose geodesic normal coordinates centred at $y=0$.
These coordinates are defined uniquely
up to a rigid rotation in $\mathbb{R}^3$.

It is known \cite[formula~(6.33)]{curl} that in geodesic normal coordinates the metric tensor reads
\begin{equation}
\label{14 September 2021 equation 1bis}
g_{\alpha\beta}(x)
=
\delta_{\alpha\beta}
-
\frac13
\Riem_{\alpha\mu\beta\nu}(0)\,x^\mu\,x^\nu
-
\frac16
(\nabla_\sigma \Riem_{\alpha\mu\beta\nu})(0)\,x^\sigma\,x^\mu\,x^\nu
+O(|x|^4)\,,
\end{equation}
where $\nabla$ is the Levi-Civita connection.
For further terms in the Taylor expansion
\eqref{14 September 2021 equation 1bis}
see \cite[formula~(6)]{Uwe Muller}.

As we are working in dimension 3, the Riemann tensor $\Riem$ is expressed in terms of the Ricci tensor $\Ric$ via the identity
\begin{multline}
\label{14 September 2021 equation 5}
\Riem_{\alpha\beta\gamma\delta}(x)
=
\Ric_{\alpha\gamma}(x)\,g_{\beta\delta}(x)
-
\Ric_{\alpha\delta}(x)\,g_{\beta\gamma}(x)
+
\Ric_{\beta\delta}(x)\,g_{\alpha\gamma}(x)
-
\Ric_{\beta\gamma}(x)\,g_{\alpha\delta}(x)
\\
+
\frac{\Sc(x)}{2}
(
g_{\alpha\delta}(x)\,g_{\beta\gamma}(x)
-
g_{\alpha\gamma}(x)\,g_{\beta\delta}(x)
)\,,
\end{multline}
where $\Sc(x):=g^{\alpha\beta}(x)\Ric_{\alpha\beta}(x)$ is scalar curvature.
The results of subsequent calculations will be expressed in terms of quantities
\[
\Ric_{\alpha\beta}(0),\
\nabla_\kappa\Ric_{\alpha\beta}(0),\
\nabla_\kappa\nabla_\lambda\Ric_{\alpha\beta}(0),\ \dots,
\qquad
\Sc(0),\
\nabla_\kappa\Sc(0),\
\nabla_\kappa\nabla_\lambda\Sc(0),\ \dots.
\]

\paragraph*{Simplification 3.}
In geodesic normal coordinates for sufficiently small $t$ we have
\begin{equation*}
\label{28 February 2024 equation 4}
[x^{(\pm)}]^\alpha(t;0,\eta)
=
\pm
\frac{\eta^\alpha}{|\eta|}\,t\,,
\qquad
[\xi^{(\pm)}]_\alpha(t;0,\eta)
=
\eta_\alpha\,,
\end{equation*}
where $|\eta|$ stands for the Euclidean norm,
i.e.~$|\eta|=(\delta^{\alpha\beta}\eta_\alpha\eta_\beta)^{1/2}$.
Here and further on we denote
$\eta^\alpha:=\delta^{\alpha\beta}\eta_\beta$.

We choose phase functions
\begin{equation}
\label{choose phase functions pm}
\varphi^{(\pm)}(t,x;0,\eta)
=
[x-x^{(\pm)}(t;0,\eta)]^\alpha
\,
[\xi^{(\pm)}]_\alpha(t;0,\eta)
=
x^\alpha\eta_\alpha\mp|\eta|t
\end{equation}
which are linear in $x$ and in $t$.
In this paper we focus on constructing the propagators
\eqref{propagator curl plus}
and
\eqref{propagator curl minus}
for small $t$, so there is no need to circumvent caustics and introduce complex-valued phase functions as in \eqref{phase functions condition 5}.

Of course, formulae
\eqref{two phase functions coincide at time zero}
and
\eqref{choose phase functions pm}
imply
\begin{equation*}
\label{choose phase functions 0}
\varphi^{(0)}(x;0,\eta)
=
x^\alpha\eta_\alpha\,.
\end{equation*}

\paragraph*{Simplification 4.}
The explicit formula for normalised eigencovectors of $[\curl_\prin]_\alpha{}^\beta(0,\eta)$ corresponding to nonzero eigenvalues reads
\begin{equation}
\label{17 June 2025 equation 1}
[b^{(\pm)}]_\beta(0,\eta)
=
\frac{1}{\sqrt{2}\,|\eta|}
\left[
\begin{pmatrix}
\eta_3\\
\pm i\eta_3\\
-(\eta_1\pm i\eta_2)
\end{pmatrix}
\pm
\frac
{i(\eta_1\pm i\eta_2)}
{\eta_3+|\eta|}
\begin{pmatrix}
-\eta_2\\
\eta_1\\
0
\end{pmatrix}
\right].
\end{equation}
Here upper sign corresponds to eigenvalue $+|\eta|$ and lower to $-|\eta|$.

Formula \eqref{17 June 2025 equation 1} works for all $\eta\ne0$ except for the south pole $\eta=(0,0,-|\eta|)$. Writing a formula for $b^{(\pm)}(0,\eta)$ so that it smoothly depends on $\eta$ for all $\eta\in\mathbb{R}^3\setminus\{0\}$ is impossible because of a topological obstruction \cite[Proposition~3.2]{obstructions}.

The eigencovectors \eqref{17 June 2025 equation 1} are defined uniquely modulo gauge transformations 
$b^{(\pm)}
\mapsto
e^{i\phi^{(\pm)}}\,b^{(\pm)}$,
where $\phi^{(\pm)}(0,\eta)$ are arbitrary real-valued scalar functions.

As in Remark~\ref{gauge transformations now sit at the source},
we define $[b^{(\pm)}_\shortparallel]_\alpha(t;0,\eta)$ as the
result of parallel transport of the covector $[b^{(\pm)}]_\alpha(0,\eta)$
along the geodesic (straight line)
$\,\displaystyle
[x^{(\pm)}]^\alpha(\tau;0,\eta)
=
\pm
\frac{\delta^{\alpha\beta}\eta_\beta}{|\eta|}\tau\,$
from $\tau=0$ to $\tau=t$.

\paragraph*{Simplification 5.}
In Appendix~\ref{The amplitude-to-symbol operator}
we defined the operator $\mathfrak{S}_{-1}\,$,
see formulae
\eqref{22 April 2026 equation 1}--\eqref{operator F j}.
This operator now reads,
for phase functions $\varphi^{(\pm)}$,
\begin{equation}
\label{22 April 2026 equation 2}
\mathfrak{S}^{(\pm)}_{-1}
=
i
\left[
\left(
\frac{\partial^2}{\partial x^\beta\partial\eta_\beta}
\pm
\frac{t}{2}
\frac{\delta^{\alpha\beta}|\eta|^2-\eta^\alpha\eta^\beta}{|\eta|^3}
\frac{\partial^2}{\partial x^\alpha\partial x^\beta}
\right)
(\,\cdot\,)
\right]_{x=x^{(\pm)}}.
\end{equation}

\paragraph*{Simplification 6.}
In view of \eqref{22 April 2026 equation 2},
formula \eqref{formula for q 4 parallel} admits a further simplification and now reads,
for $\aleph=\pm\,$,
\begin{multline*}
q^{(\pm)}_{0,\shortparallel}
=
-i
\,
\overline
{	
[b^{(\pm)}_\shortparallel]_\alpha
}
\,g^{\alpha\beta}(x^{(\pm)})\,
\frac{\dr [b^{(\pm)}_\shortparallel]_\beta}{\dr t}
\\
+
\overline
{	
[b^{(\pm)}_\shortparallel]_\kappa
}
\,g^{\kappa\lambda}(x^{(\pm)})
\left[
\left(
\frac{\partial^2}{\partial x^\beta\partial\eta_\beta}
\pm
\frac{t}{2}
\frac{\delta^{\alpha\beta}|\eta|^2-\eta^\alpha\eta^\beta}{|\eta|^3}
\frac{\partial^2}{\partial x^\alpha\partial x^\beta}
\right)
	E_\lambda{}^{\mu\nu}(x)\,
	\eta_\nu\,
	[b^{(\pm)}_\shortparallel]_\mu
	\right]_{x=x^{(\pm)}}.
\end{multline*}

\paragraph*{Simplification 7.}
We will seek the homogeneous components
\eqref{symbols of propagators U aleph}
of the symbols
of the time-dependent oscillatory integrals
\eqref{time-dependent oscillatory integrals}
in the form of Taylor expansions in $t$.

\subsection{Evaluation of Weyl coefficients}
\label{Evaluation of Weyl coefficients}

Formula~\eqref{local counting function mollified} now reads
\begin{multline}
\label{8 April 2026 equation 5}
(2\pi)^{-4}
\int
e^{\pm it(\lambda-|\eta|)}\,
[v^{(\pm)}]_\alpha{}^\alpha(t;0,\eta)\,
\widehat\mu(t)\,\chi(|\eta|)\,
\dr\eta\,\dr t
\\
=
c_2^\pm(0)\,\lambda^2
+
c_1^\pm(0)\,\lambda
+
c_0^\pm(0)
+
c_{-1}^\pm(0)\,\lambda^{-1}
+\dots
\quad
\text{as}
\quad
\lambda\to+\infty\,.
\end{multline}
Here, in order to derive the asymptotic expansion, one needs to decompose
the $v^{(\pm)}$ into components positively homogeneous in $\eta$ and
write the latter as Taylor expansions in $t$.
The function
$\chi:\mathbb{R}\to\mathbb{R}$
is a smooth cut-off such that
$\chi(r)=0$ for $r\le1/2$
and
$\chi(r)=1$ for $r\ge1$,
as in \cite[formula~(B.16)]{wave}.

The way to handle the integral in the LHS of \eqref{8 April 2026 equation 5}
is to integrate over Euclidean spheres $|\eta|=r$ first, which reduces the
derivation of the asymptotic expansion to the evaluation of integrals
\begin{equation}
\label{9 April 2026 equation 1}
\int_{-\infty}^{+\infty}
\left[
\int_{-\infty}^{+\infty}
t^n\,e^{\pm i(\lambda-r)t}\,
\widehat\mu(t)\,\dr t
\right]
r^m\,
\chi(r)\,\dr r\,,
\quad
m=2,1,0,-1,\dots,
\quad
n=0,1,2,\dots.
\end{equation}
Note that for any fixed $\lambda\in\mathbb{R}$ the quantity inside square brackets in
\eqref{9 April 2026 equation 1}
decays superpolynomially in $r$ as $r\to+\infty$.
Hence, the integral
\eqref{9 April 2026 equation 1}
makes sense.

Observe that
\[
t\,e^{\pm i(\lambda-r)t}=\pm i\frac\partial{\partial r}e^{\pm i(\lambda-r)t}\,,
\]
so integrating by parts in \eqref{9 April 2026 equation 1} we reduce,
modulo $O(\lambda^{-\infty})$, the evaluation of
\eqref{9 April 2026 equation 1}
to the evaluation of
\begin{equation*}
\label{9 April 2026 equation 2}
\int_{-\infty}^{+\infty}
\left[
\int_{-\infty}^{+\infty}
e^{\pm i(\lambda-r)t}\,
\widehat\mu(t)\,\dr t
\right]
r^m\,
\chi(r)\,\dr r\,,
\end{equation*}
for $m=2,1,0,-1,\dots$.
The latter integral can be equivalently rewritten as
\begin{equation*}
\label{10 April 2026 equation 1}
2\pi
\int_{-\infty}^{+\infty}
\mu(\lambda-r)\,
r^m\,
\chi(r)\,\dr r\,.
\end{equation*}

Our mollifier $\mu$ has the properties
\begin{equation}
\label{10 April 2026 equation 2}
\int_{-\infty}^{+\infty}\mu(s)\,\dr s=1\,,
\qquad
\int_{-\infty}^{+\infty}s^k\,\mu(s)\,\dr s=0\,,
\qquad
k=1,2,\dots.
\end{equation}
Arguing as in \cite[Lemma~3.2.1]{fang_PhD},
it is easy to see that \eqref{10 April 2026 equation 2} implies
\begin{equation}
\label{9 April 2026 equation 3}
\int_{-\infty}^{+\infty}
\mu(\lambda-r)\,
r^m\,
\chi(r)\,\dr r
=
\lambda^m
+
O(\lambda^{-\infty})
\quad
\text{as}
\quad
\lambda\to+\infty\,,
\end{equation}
for all $m\in\mathbb{Z}$.
Compare with \cite[page~1771]{wave}.

\section{The first two Weyl coefficients}
\label{The first two Weyl coefficients}

Application of the algorithm from
Section~\ref{The algorithm}
with account of the simplifications listed in
subsection~\ref{Simplifications}
gives
\begin{equation}
\label{8 April 2026 equation 1}
b^{(\pm)}_\shortparallel(t;0,\eta)
=
b^{(\pm)}(0,\eta)
+
O(t^2)\,,
\end{equation}
\begin{equation}
\label{8 April 2026 equation 2}
q^{(\pm)}_{0,\shortparallel}(t;0,\eta)
=
O(t)\,,
\end{equation}
\begin{equation}
\label{8 April 2026 equation 3}
v^{(\pm)}_0(t;0,\eta)
=
P^{(\pm)}(0,\eta)
+
O(t^2)\,,
\end{equation}
\begin{equation}
\label{8 April 2026 equation 4}
v^{(\pm)}_{-1}(t;0,\eta)
=
O(t)
\end{equation}
as $t\to0$.
Recall that by $P^{(\aleph)}$, $\aleph\in\{+,-,0\}$, we denote the the principal symbols of the projection operators $P_\aleph$ from \cite{curl}.
Formulae \eqref{8 April 2026 equation 1}--\eqref{8 April 2026 equation 4}
are an immediate consequence of the fact that there is no linear term in the expansion
\eqref{14 September 2021 equation 1bis}.


\

Formulae
\eqref{8 April 2026 equation 3}
and
\eqref{8 April 2026 equation 4}
allow us to evaluate the first two local Weyl coefficients
in accordance with the procedure described in
subsection~\ref{Evaluation of Weyl coefficients}.
We get
\begin{equation*}
\label{11 April 2026 equation 1}
c_2^\pm(0)
=
\frac{1}{2\pi^2}\,,
\end{equation*}
\begin{equation}
\label{11 April 2026 equation 2}
c_1^\pm(0)
=
0\,,
\end{equation}
which imply formulae~\eqref{Theorem main result equation 0} and~\eqref{Theorem main result equation 1} from Theorem~\ref{Theorem main result}. Note that formula \eqref{11 April 2026 equation 2} also follows from the results presented in
Appendix~\ref{Scaling the Riemannian manifold}. 

\section{Proof of Theorems~\ref{theorem 1}, \ref{Theorem main result 2} and~\ref{Theorem main result 3}}
\label{Proof of Theorems on Weyl asymptotics}

\begin{proof}[Proof of Theorem~\ref{theorem 1}]
In Sections \ref{The algorithm}--\ref{The first two Weyl coefficients}
we established that
\begin{equation}
\label{Proof of Theorems on Weyl asymptotics equation 1}
(N_\pm'*\mu)(y;\lambda)
=
\frac{1}{2\pi^2}\,\lambda^2
+
O(1)
\qquad \text{as}\qquad \lambda \to + \infty
\end{equation}
with remainder uniform over $y\in M$.
We also have
\begin{equation}
\label{Proof of Theorems on Weyl asymptotics equation 2}
(N_\pm'*\mu)(y;\lambda)
=
O(\lambda^{-\infty})
\qquad \text{as}\qquad \lambda \to - \infty
\end{equation}
with remainder also uniform over $y\in M$.
Formulae
\eqref{Proof of Theorems on Weyl asymptotics equation 1}
and
\eqref{Proof of Theorems on Weyl asymptotics equation 2}
imply
\begin{equation}
\label{Proof of Theorems on Weyl asymptotics equation 3}
(N_\pm*\mu)(y;\lambda)
=
\frac{1}{6\pi^2}\,\lambda^3
+
O(\lambda)
\qquad \text{as}\qquad \lambda \to + \infty\,.
\end{equation}
According to \cite[Corollary~B.2.2]{SaVa}
formulae
\eqref{Proof of Theorems on Weyl asymptotics equation 1}
and
\eqref{Proof of Theorems on Weyl asymptotics equation 3}
imply
\eqref{theorem 1 equation 1}.
\end{proof}

\begin{proof}[Proof of Theorem~\ref{Theorem main result 2}]
Let $\widehat{\gamma}:\mathbb{R}\to \mathbb{C}$,
be an arbitrary smooth compactly supported function such that
$\operatorname{supp} \widehat{\gamma}\subset(0,+\infty)$.
Arguing as in the proof of
\cite[Theorem 4.4.9]{SaVa}
we get
\begin{equation}
\label{Proof of Theorems on Weyl asymptotics equation 4}
(N_\pm'*\gamma)(y;\lambda)
=
o(\lambda^2)
\qquad \text{as}\qquad \lambda \to + \infty\,.
\end{equation}
According to \cite[Theorem~B.5.1]{SaVa}
formulae
\eqref{Proof of Theorems on Weyl asymptotics equation 1},
\eqref{Proof of Theorems on Weyl asymptotics equation 3}
and
\eqref{Proof of Theorems on Weyl asymptotics equation 4}
imply
\eqref{Theorem main result 2 equation 1}.
\end{proof}

\begin{proof}[Proof of Theorem~\ref{Theorem main result 3}]
Formulae
\eqref{Proof of Theorems on Weyl asymptotics equation 1}
and
\eqref{Proof of Theorems on Weyl asymptotics equation 2}
imply
\begin{equation}
\label{Proof of Theorems on Weyl asymptotics equation 5}
(N_\pm'*\mu)(\lambda)
=
\frac{\operatorname{Vol}M}{2\pi^2}\,\lambda^2
+
O(1)
\qquad \text{as}\qquad \lambda \to + \infty\,,
\end{equation}
\begin{equation}
\label{Proof of Theorems on Weyl asymptotics equation 6}
(N_\pm'*\mu)(\lambda)
=
O(\lambda^{-\infty})
\qquad \text{as}\qquad \lambda \to - \infty\,.
\end{equation}

Formulae
\eqref{Proof of Theorems on Weyl asymptotics equation 5}
and
\eqref{Proof of Theorems on Weyl asymptotics equation 6}
imply
\begin{equation}
\label{Proof of Theorems on Weyl asymptotics equation 7}
(N_\pm*\mu)(\lambda)
=
\frac{\operatorname{Vol}M}{6\pi^2}\,\lambda^3
+
O(\lambda)
\qquad \text{as}\qquad \lambda \to + \infty\,.
\end{equation}

Let $\widehat{\gamma}$ be as in the proof of Theorem~\ref{Theorem main result 2}.
Arguing as in the proof of
\cite[Theorem 4.4.1]{SaVa}
we get
\begin{equation}
\label{Proof of Theorems on Weyl asymptotics equation 8}
(N_\pm'*\gamma)(\lambda)
=
o(\lambda^2)
\qquad \text{as}\qquad \lambda \to + \infty\,.
\end{equation}
According to \cite[Theorem~B.5.1]{SaVa}
formulae
\eqref{Proof of Theorems on Weyl asymptotics equation 5},
\eqref{Proof of Theorems on Weyl asymptotics equation 7}
and
\eqref{Proof of Theorems on Weyl asymptotics equation 8}
imply
\eqref{Theorem main result 3 equation 1}.
\end{proof}

\section{The third Weyl coefficients}
\label{The third Weyl coefficient}

Formula
\eqref{14 September 2021 equation 1bis}
tells us that the third Weyl coefficients $c_0^\pm(0)$ will be proportional to
$\Riem_{\alpha\mu\beta\nu}(0)$, the Riemann curvature tensor at the origin of
our geodesic normal coordinate system.
And there is only one scalar that can be formed out of the
Riemann curvature tensor  --- scalar curvature.
See also Appendix~\ref{Scaling the Riemannian manifold}. Thus, without loss of generality, we introduce the following additional assumption.

\paragraph*{Simplification 8.}
Henceforth, we simplify calculations by assuming
that curvature at the origin is purely scalar, i.e.~that
\begin{equation}
\label{11 April 2026 equation 3}
\Riem_{\alpha\beta\gamma\delta}(0)
=
-
\frac{\Sc(0)}{6}
(
\delta_{\alpha\delta}\delta_{\beta\gamma}
-
\delta_{\alpha\gamma}\delta_{\beta\delta}
)\,.
\end{equation}
Formula
\eqref{11 April 2026 equation 3}
is obtained by substituting
\begin{equation*}
\label{29 April 2026 equation 1}
\Ric_{\alpha\beta}(0)=\frac13\delta_{\alpha\beta}\Sc(0)
\end{equation*}
and
$g_{\alpha\beta}(0)=\delta_{\alpha\beta}$
into
\eqref{14 September 2021 equation 5}.

\

Substituting
\eqref{11 April 2026 equation 3}
into
\eqref{14 September 2021 equation 1bis}
we get
\begin{equation}
\label{12 April 2026 equation 1}
g_{\alpha\beta}(x)
=
\delta_{\alpha\beta}
-
\frac{\Sc(0)}{18}
(
\delta_{\alpha\beta}\delta_{\mu\nu}
-
\delta_{\alpha\mu}\delta_{\beta\nu}
)\,x^\mu\,x^\nu
+O(|x|^3)\,.
\end{equation}
Formula
\eqref{12 April 2026 equation 1}
is the starting point in the calculation of the third Weyl coefficient.

An additional simplification is the observation that in view of spherical symmetry it suffices to perform calculations only at the north pole, i.e.~at
\begin{equation}
\label{12 April 2026 equation 2}
\eta_\alpha
=
\begin{pmatrix}
0\\
0\\
Z
\end{pmatrix},
\qquad
Z>0\,.
\end{equation}

Application of the algorithm from
Section~\ref{The algorithm}
with account of simplifications listed in
sub\-section~\ref{Simplifications}
as well as
\eqref{12 April 2026 equation 1}
and
\eqref{12 April 2026 equation 2}
gives
\begin{equation}
\label{12 April 2026 equation 3}
[b^{(\pm)}_\shortparallel]_\alpha(t;0,\eta)
=
\frac{1}{\sqrt{2}}
\left(
1
-
\frac{\Sc(0)\,t^2}{36}
\right)
\begin{pmatrix}
1\\
\pm i\\
0
\end{pmatrix}
+
O(t^3)\,,
\end{equation}
\begin{equation}
\label{12 April 2026 equation 4}
q^{(\pm)}_{0,\shortparallel}(t;0,\eta)
=
\frac{i\Sc(0)\,t}{18}
+
O(t^2)\,,
\end{equation}
\begin{equation}
\label{12 April 2026 equation 5}
[v^{(\pm)}_0]_\alpha{}^\beta(t;0,\eta)
=
\frac{1}{2}
\begin{pmatrix}
1&\mp i&0\\
\pm i&1&0\\
0&0&0
\end{pmatrix}
+
O(t^3)\,,
\end{equation}
\begin{equation}
\label{12 April 2026 equation 6}
[v^{(\pm)}_{-1}]_\alpha{}^\beta(t;0,\eta)
=
\frac{\Sc(0)\,t}{36Z}
\begin{pmatrix}
\mp i&-1&0\\
1&\mp i&0\\
0&0&\mp 2i
\end{pmatrix}
+
O(t^2)\,,
\end{equation}
and
\begin{equation}
\label{12 April 2026 equation 7}
[v^{(\pm)}_{-2}]_\alpha{}^\beta(t;0,\eta)
=
\frac{\Sc(0)}{36Z^2}
\begin{pmatrix}
0&\pm i&0\\
\mp i&0&0\\
0&0&-2
\end{pmatrix}
+
O(t)
\end{equation}
as $t\to0$.
Compare with formulae
\eqref{8 April 2026 equation 1}--\eqref{8 April 2026 equation 4}.
In the right-hand sides of formulae
\eqref{12 April 2026 equation 5}--\eqref{12 April 2026 equation 7}
we employed matrix notation
with the first index $\alpha$ enumerating rows and the second $\beta$ enumerating columns.

Formulae
\eqref{12 April 2026 equation 5}--\eqref{12 April 2026 equation 7} were written at the north pole \eqref{12 April 2026 equation 2}. For an arbitrary $\eta\in\mathbb{R}^3\setminus\{0\}$ they read
\begin{equation}
\label{8 May 2026 equation 1}
v^{(\pm)}_0(t;0,\eta)
=
P^{(\pm)}(0,\eta)
+
O(t^3)\,,
\end{equation}
\begin{equation}
\label{8 May 2026 equation 2}
v^{(\pm)}_{-1}(t;0,\eta)
=
\mp
\frac{i\Sc(0)\,t}{18|\eta|}
\left[
P^{(\pm)}(0,\eta)
+
P^{(0)}(0,\eta)
\right]
+
O(t^2)\,,
\end{equation}
\begin{equation}
\label{8 May 2026 equation 3}
v^{(\pm)}_{-2}(t;0,\eta)
=
-
\frac{\Sc(0)}{36|\eta|^2}
\left[
P^{(\pm)}(0,\eta)
-
P^{(\mp)}(0,\eta)
+
2
P^{(0)}(0,\eta)
\right]
+
O(t)\,.
\end{equation}

Formulae
\eqref{8 May 2026 equation 1}--\eqref{8 May 2026 equation 3}
allow us to evaluate the third Weyl coefficient
in accordance with the procedure described in
subsection~\ref{Evaluation of Weyl coefficients}.
We get
\begin{equation*}
\label{8 May 2026 equation 4}
c_0^\pm(0)
=
-\frac{1}{12\pi^2}\,\Sc(0)\,,
\end{equation*}
which implies formula~\eqref{Theorem main result equation 2} from Theorem~\ref{Theorem main result}.

\

\begin{remark}
	Let us compare the first three local Weyl coefficients of \(\curl\), \(c^\pm_n(y)\),
	\(n=2,1,0\), with those of the Laplace--Beltrami operator \(\Delta\) on the same Riemannian 3-manifold \((M,g)\). A similar comparison involving the massless Dirac operator was undertaken in \cite[Remark~8.2]{dirac}.
	
	Let \(N(y;\lambda)\) be the local counting function of \(\sqrt{-\Delta}\). Then, as
	\(\lambda\to+\infty\),
\begin{equation*}
\label{Laplace-Beltrami expansion}
	\left(N'\ast \mu \right)(y;\lambda) = c_2(y)\,\lambda^2 + c_1(y)\,\lambda + c_0(y) + O(\lambda^{-1})\,.
\end{equation*}
Hence, \cite[Theorem~B.2]{wave} and formulae
\eqref{Theorem main result equation 0}--\eqref{Theorem main result equation 2}
imply
	\begin{equation*}
	\label{comparing Weyl coefficients}
c^\pm_2(y)=c_2(y)\,,
\qquad
c^\pm_1(y)=c_1(y)=0\,,
\qquad
c^\pm_0(y)=-2\,c_0(y)\,.
\end{equation*}
\end{remark}

\

\begin{remark}
\label{remark on right symbols of projections}

Setting $t=0$ in formulae
\eqref{8 May 2026 equation 1}--\eqref{8 May 2026 equation 3}
we arrive at the following expression for the right symbols of projection operators $P_\pm$:
\begin{equation}
\label{4 June 2026 equation 1}
P^{(\pm)}(0,\eta)
-
\frac{\Sc(0)}{36|\eta|^2}
\left[
P^{(\pm)}(0,\eta)
-
P^{(\mp)}(0,\eta)
+
2
P^{(0)}(0,\eta)
\right]
+O(|\eta|^{-3})\,.
\end{equation}
As we are working in geodesic normal coordinates, the right symbols of
projection operators $P_+$, $P_-$ and $P_0$ add up to the identity matrix.
Hence, formula \eqref{4 June 2026 equation 1} implies the following expression for the right symbol of the projection operator $P_0$
\begin{equation}
\label{5 June 2026 equation 1}
\left(
1
+
\frac{\Sc(0)}{9|\eta|^2}
\right)
P^{(0)}(0,\eta)
+
O(|\eta|^{-3})\,.
\end{equation}
Recall that according to \cite[formula~(1.8)]{curl} we have
$[P^{(0)}]_\alpha{}^\beta(0,\eta)=|\eta|^{-2}\eta_\alpha\eta^\beta$,
hence, the matrix-function \eqref{5 June 2026 equation 1} is rank 1.
This is in agreement with Remark~\ref{remark P0 is a rank 1 matrix-function}.

An alternative way of deriving formulae
\eqref{4 June 2026 equation 1}
and
\eqref{5 June 2026 equation 1}
is by means of the algorithm from \cite[subsection.~4.3]{part1}.
See also \cite[Proposition~5.3]{curl} for a detailed explanation of
how the algorithm from \cite[subsection.~4.3]{part1} works in the specific case
of the operator $\curl$. The advantage of this alternative approach is that it allows
one to calculate the full right symbols of projection operators
$P_\aleph\,$, $\aleph\in\{+,-,0\}$, without solving transport equations.

\end{remark}

\section{Higher Weyl coefficients}
\label{Higher Weyl coefficients}

\subsection{The fourth Weyl coefficients}
\label{The fourth Weyl coefficients}

Results presented in
Appendix~\ref{Scaling the Riemannian manifold}
immediately imply that the fourth local Weyl coefficients, $c_{-1}^\pm(y)$, are zero, so that we have formula~\eqref{Theorem main result equation 3} from Theorem~\ref{Theorem main result}.

\subsection{The fifth Weyl coefficients}
\label{The fifth Weyl coefficients}

Results presented in
Appendix~\ref{Scaling the Riemannian manifold} and~Theorem~\ref{theorem symmetries of Weyl coefficients}
imply that the fifth local Weyl coefficients, $c_{-2}^\pm(y)$, are symmetric,
$c_{-2}^+(y)=c_{-2}^-(y)\,$,
and are linear combinations of three scalar geometric invariants
$\Sc^2(y)$, $\|\oRic(y)\|^2$ and $(\Delta\Sc)(y)\,$.

Let us consider the special case of the round sphere,
see subsection~\ref{The round sphere}.
Formula \eqref{Arguing as in formula (9.19) dirac}
tells us that $\Sc^2(y)$ does not appear in the expression for the
fifth local Weyl coefficients,
which leaves us with two
scalar geometric invariants,
$\|\oRic(y)\|^2$ and $(\Delta\Sc)(y)\,$.

Finally, let us consider the special case of the Berger sphere,
see subsection~\ref{The Berger sphere}.
Examination of formulae
\eqref{local counting function expansion Berger}
and
\eqref{Berger traceless Ricci squared}
yields formula~\eqref{Theorem main result equation 4} from Theorem~\ref{Theorem main result}.

\subsection{The sixth Weyl coefficients}
\label{The sixth Weyl coefficients}

This is the first instance where we observe spectral asymmetry in spectral asymptotics.

Results presented in
Appendix~\ref{Scaling the Riemannian manifold} and~Theorem~\ref{theorem symmetries of Weyl coefficients}
imply that the sixth local Weyl coefficients, $c_{-3}^\pm(y)$, are antisymmetric,
$c_{-3}^+(y)=-c_{-3}^-(y)\,$,
and read
\begin{equation}
\label{The sixth Weyl coefficients equation 1}
c_{-3}^\pm(y)
=
\pm\,C\,
E^{\alpha\beta\gamma}(y)\,\oRic_{\alpha\mu}(y)\,\nabla_\beta\,\oRic_\gamma{}^\mu(y)
\,,
\end{equation}
where $C$ is some universal constant.
Examination of formulae
\eqref{local counting function expansion Berger}
and
\eqref{Berger pseudoscalar invariant}
gives us
\begin{equation}
\label{The sixth Weyl coefficients equation 2}
c_{-3}^+(y)-c_{-3}^-(y)
=
-\frac{1}{60\pi^2}\,
E^{\alpha\beta\gamma}(y)\,\oRic_{\alpha\mu}(y)\,\nabla_\beta\,\oRic_\gamma{}^\mu(y)\,.
\end{equation}
Formulae
\eqref{The sixth Weyl coefficients equation 1}
and
\eqref{The sixth Weyl coefficients equation 2}
imply that $\displaystyle C=-\frac{1}{120\pi^2}\,$.
This proves formula~\eqref{Theorem main result equation 5} from Theorem~\ref{Theorem main result}.

\section{Examples}
\label{Examples}

In this section, we discuss two examples for which the spectrum of $\curl$ is known explicitly: the round 3-sphere and the Berger sphere. Although the round sphere is a special case of the Berger sphere, we have opted to treat the two separately in view of the substantial difference in their technical complexity. In a way, the former is a warm-up for the latter.

\subsection{The round sphere}
\label{The round sphere}
According to
\cite[formulae (D.11) and (D.12)]{curl}
and
\cite[Theorem~5.2]{baer_curl},
the eigenvalues of $\curl$ on the round sphere $(\mathbb{S}^3,g_1)$ are
\begin{equation}
\label{spectrum curl 3 sphere eigenvalues}
\pm k\,, \qquad k=2,3,\dots,
\end{equation}
with multiplicity 
\begin{equation}
\label{spectrum curl 3 sphere multiplicity}
k^2-1\,.
\end{equation}
In particular, the spectrum is symmetric about $0$.

Formulae~\eqref{spectrum curl 3 sphere eigenvalues}, \eqref{spectrum curl 3 sphere multiplicity}, \cite[formulae~(8.2)~and~(8.3)]{dirac}, and the fact that
$\operatorname{Vol}_{g_1}\mathbb{S}^3 = 2\pi^2$
imply
\begin{align*}
[N_\pm'\ast{\mu}](y,\lambda)
&=
\mathcal{F}^{-1}_{t\to \lambda}
\left[
\frac1{2\pi^2}
\sum_{k=2}^\infty e^{-ikt}\,(k^2-1)\,\widehat{\mu}(t)
\right]
\\
&=
\frac{1}{4\pi^3}
\sum_{k=2}^{+\infty}
\int_{-\infty}^{+\infty}
e^{it(\lambda-k)}\,(k^2-1)\,\widehat{\mu}(t)\,\dr t
\\
&=
\frac{1}{4\pi^3}
\sum_{k=-\infty}^{+\infty}
\int_{-\infty}^{+\infty}
e^{it(\lambda-k)}\,(k^2-1)\,\widehat{\mu}(t)\,\dr t
+
O(\lambda^{-\infty})
\\
&=
\frac{\lambda^2-1}{2\pi^2}
+
O(\lambda^{-\infty})
=
\frac{1}{2\pi^2}
\,
\lambda^2
-
\frac{1}{12\pi^2}\,\Sc(y)
+
O(\lambda^{-\infty})\,,
\stepcounter{equation}\tag{\theequation}\label{Arguing as in formula (9.19) dirac}
\end{align*}
where $\Sc(y)=6$ is the scalar curvature of the standard round sphere.

It is noteworthy that only two local Weyl coefficients do not vanish for the round sphere, namely~$c^\pm_2(y)$ and~$c^\pm_0(y)$.
In principle, there could have been more non-vanishing coefficients, involving powers of scalar curvature. In particular, \eqref{Arguing as in formula (9.19) dirac} confirms that a term proportional $\Sc^2(y)$ does not appear in the fifth local Weyl coefficient of $\curl$.

\subsection{The Berger sphere}
\label{The Berger sphere}

For the definition and a discussion of the elementary properties of the Berger sphere $(\mathbb{S}^3,g_a)$ with parameter $a>0$ we refer the reader to \cite[Appendices~C~and~D]{curl}. As per \cite[Theorem~D.1]{curl}, the spectrum of \(\curl\) on the Berger sphere is the (disjoint) union of the following four sequences of eigenvalues:
\begin{enumerate}
	\item
	Eigenvalues
	\begin{equation}
	\label{eq:lambda_I}
	\frac{n}{a}\,,\qquad
	n=2,3,\dotsc,
	\end{equation}
	with multiplicity \(2n-2\).
	\item
	Eigenvalues
	\begin{equation}
	\label{eq:lambda_II}
	\frac{n+2(a^2-1)}{a}\,,\qquad
	n=2,3,\dotsc.
	\end{equation}
	Here the multiplicity is as follows.
	\begin{enumerate}
		\item
		If $n=2$ the multiplicity is $1$.
		\item
		If $n=3,4,\dots$ the multiplicity is $2n-2$.
	\end{enumerate}
	\item
	Eigenvalues
	\begin{equation}
	\label{eq:lambda_III}
	a
	+
	\sqrt{
		a^2
		+
		n(n+2)
		+
		4\left( 
		a^{-2}-1
		\right)
		\left(
		q-\frac n 2
		\right)^2
	}\ ,
	\end{equation}
	with \(n=2,3,\dotsc\) and \(q=1,\dotsc,n-1\). Here the multiplicity is \(n+1\) independently of \(q\).
	\item
	Eigenvalues
	\begin{equation}
	\label{eq:lambda_IV}
	a
	-
	\sqrt{
		a^2
		+
		n(n+2)
		+
		4\left( 
		a^{-2}-1
		\right)
		\left(
		q - \frac n 2
		\right)^2
	}\ ,
	\end{equation}
	with \(n=2,3,\dotsc\) and \(q=1,\dotsc,n-1\). Here the multiplicity is \(n+1\) independently of \(q\).
\end{enumerate}
Sequences I, II, and III form the positive spectrum, whereas Sequence IV constitutes the negative spectrum. If \(a=1\), the metric \(g_1\) is the round metric and the spectrum reduces to \eqref{spectrum curl 3 sphere eigenvalues}, \eqref{spectrum curl 3 sphere multiplicity}. Note that the labelling of eigenvalues in the third and fourth families differs here from the presentation in \cite[Theorem~D.1]{curl}.

Let us introduce local counting functions \(N_\mathrm{I}(y;\lambda)\), \(N_\mathrm{II}(y;\lambda)\), and \(N_\mathrm{III}(y;\lambda)\) corresponding to the three families of eigenvalues \eqref{eq:lambda_I}--\eqref{eq:lambda_III}.

Arguing as in the previous subsection and using the fact that $\operatorname{Vol}_{g_a}\mathbb{S}^3 = 2\pi^2a$ one obtains for the first two families
\begin{align*}
\left(N_\mathrm{I}'\ast{\mu}\right)(y;\lambda)
&=
\mathcal{F}^{-1}_{t\to \lambda}
\left[ \frac{1}{2\pi^2a}
\sum_{n=2}^{+\infty}(2n-2) \, e^{- it \frac n a }\,\widehat{\mu}(t)
\right]
\\
&= \frac{1}{2\pi^3} \sum_{n=-\infty}^{+\infty}\int_{-\infty}^{+\infty} (n-1)\, e^{i (a\lambda - n)t}\, \widehat\mu(at)\,\dr t + O(\lambda^{-\infty}) \\
&= \frac{a\lambda-1}{\pi^2} + O(\lambda^{-\infty}) = \frac{a}{\pi^2}\,\lambda - \frac{1}{\pi^2} + O(\lambda^{-\infty}) \stepcounter{equation}\tag{\theequation}\label{4 January 2026 equation 9}
\end{align*}
and
\begin{align*}
\left(N_\mathrm{II}'\ast{\mu}\right)(y;\lambda)
&=
\mathcal{F}^{-1}_{t\to \lambda}
\left[ \frac{1}{2\pi^2a}
\sum_{n=2}^\infty(2n-2) \, e^{- it\frac{n+2(a^2-1)}{a} }\,\widehat{\mu}(t)
\right]+ O(\lambda^{-\infty}) \\
&= \frac{1}{2\pi^3} \sum_{n=-\infty}^\infty\int_{-\infty}^{+\infty}(n-1)\, e^{i (a\lambda - 2(a^2-1) - n )t}\,\widehat\mu(at)\,\dr t + O(\lambda^{-\infty})\\
&= \frac {a\lambda - 2(a^2-1) - 1}{\pi^2}  + O(\lambda^{-\infty})
= \frac{a}{\pi^2}\,\lambda + \frac{1-2a^2}{\pi^2} + O(\lambda^{-\infty}) \,.\stepcounter{equation}\tag{\theequation}\label{4 January 2026 equation 10}
\end{align*}

As to the third and fourth families, their examination requires considerably more effort owing to the radical expression appearing in \eqref{eq:lambda_III} and \eqref{eq:lambda_IV} and the presence of an additional summation over the index \(q=1,\dotsc,n-1\).

Let us introduce the auxiliary function of the variables \((u;t)\in[0,1]\times \mathbb{R}\)
\begin{equation*}
	f_n(u;t) :=  e^{-i t \sqrt{a^2 + 2n + n^2\varphi_1^2(u)}} \qquad \text{with}\quad \varphi_1(u) := \left[1 + 4(a^{-2}-1)\left(u-\frac 1 2\right)^2\right]^{1/2}\,,
\end{equation*}
so that
\begin{equation}
\label{10 April 2026 equation 4}
\left(N_\mathrm{III}'\ast{\mu}\right)(y;\lambda)
= \frac{1}{4\pi^3a}\sum_{n=2}^{+\infty} \int_{-\infty}^{+\infty}(n+1)\, e^{i(\lambda - a)t}\sum_{q=1}^{n-1} f_n\left(\frac q n;t\right) \widehat{\mu}(t)\,\dr t\,.
\end{equation}

Note that the mollified derivative of the \emph{negative} local counting function \(N_-(y;\lambda)\) admits an expansion identical to \eqref{10 April 2026 equation 4}, but with \(e^{i(\lambda - a)t}\) replaced by \(e^{i(\lambda + a)t}\). It can thus be treated in essentially the same manner, with the appropriate adjustments, and we omit the details.

In order to do away with the sum over \(q\) in \eqref{10 April 2026 equation 4}, 
we approximate it with its Euler--Maclaurin expansion. For fixed \(n\geq 2\) and \(t\) in the support of \(\widehat\mu\), we have
\begin{equation}
\label{10 April 2026 equation 5 infinite}
\sum_{q=1}^{n-1} f_n\!\left(\frac q n;t\right)
=
\mathcal B_n(t)-\mathcal T_n(t)+\mathcal I^{(0)}_n(t)
-\mathcal I^{(1)}_n(t) + \mathcal I^{(3)}_n(t) - \mathcal R^{(4)}_n(t)\,,
\end{equation}
where
\begin{equation}
\label{Euler-Maclaurin first few terms}
	\mathcal{B}_n(t) := n\int_0^1 f_n\!\left(u;t\right)\dr u\,, \qquad \mathcal{T}_n(t) := 2\int_0^1 f_n\!\left(\frac q n;t\right) \dr q\,, \qquad \mathcal{I}^{(0)}_n(t) := f_n\!\left(\frac 1 n;t\right), 
\end{equation}
\begin{equation}
\label{Euler-Maclaurin endpoint derivative terms}
	\mathcal I^{(1)}_n(t) := \frac{1}{6n}\,
	f_n'\!\left(\frac 1 n;t\right), \qquad \mathcal I^{(3)}_n(t) := \frac{1}{360n^3}\,
	f_n'''\!\left(\frac 1 n;t\right)\,,
\end{equation}
and
\begin{equation}
\label{Euler-Maclaurin remainder term}
	\mathcal R^{(4)}_n(t) = \frac{1}{24n^3}\int_{1/n}^{1-1/n} f_n''''(u;t)\,B_4(\{nu\})\,\dr u\,.
\end{equation}
Here \(B_k(x)\) denotes the \(k\)-th Bernoulli polynomial and \(\{x\}\) denotes the fractional part of the real number \(x\). In formulae \eqref{Euler-Maclaurin endpoint derivative terms} and \eqref{Euler-Maclaurin remainder term} primes denote differentiations with respect to \(u\).

Inserting formula \eqref{10 April 2026 equation 5 infinite} into \eqref{10 April 2026 equation 4}, one obtains a decomposition of \(\left(N_\mathrm{III}'\ast{\mu}\right)(y;\lambda)\) as a finite sum of certain functions of \(\lambda\):
\begin{equation*}
\label{certain functions of lambda}
\left(N_\mathrm{III}'\ast{\mu}\right)(y;\lambda)
= B(\lambda) - T(\lambda) + I^{(0)}(\lambda) - I^{(1)}(\lambda) + I^{(3)}(\lambda) - R^{(4)}(\lambda)\,.
\end{equation*}
One can then examine each contribution individually, making use of the following facts.
\begin{enumerate}[(i)]
	\item One has the Taylor expansion
	\begin{equation}
	\label{root expansion}
	\sqrt{a^2 + 2n + n^2\varphi_1^2(u)} = n\,\varphi_1(u) + \varphi_0(u) + \frac{\varphi_{-1}(u)}{n} + \frac{\varphi_{-2}(u)}{n^2} + \cdots,
	\end{equation}
	which is valid as \(n\to+\infty\) uniformly in \(u\). In \eqref{root expansion} the coefficients \(\varphi_j(u)\), \(j=0,-1,\dotsc,\) are certain rational functions of \(a\) and \(\varphi_1(u)\) --- for example, \(\varphi_0(u) = [\varphi_1(u)]^{-1}\). 
	\item Let \(\alpha \in C^\infty([0,1])\) be bounded away from \(0\), that is, \(\alpha_*\coloneqq \min_{[0,1]}\alpha > 0\), and suppose \(\operatorname{supp}\widehat\mu \subset (-2\pi\alpha_*,2\pi\alpha_*)\). If \(m\in\mathbb{Z}\) and \(\ell\in\mathbb{N}\), then we have the asymptotic expansion
	\begin{equation}
	\label{distributional identity}
	\frac{1}{2\pi}\sum_{n=2}^{+\infty} \int_{-\infty}^{+\infty} n^m\,(it)^\ell\, e^{it\left(\lambda - \frac n {\alpha(u)}\right)}\,\widehat\mu(t)\,\dr t = (m)_\ell\,
	[\alpha(u)]^{m+1}\,\lambda^{m-\ell} + O(\lambda^{-\infty})
	\end{equation}
	 as \(\lambda\to+\infty\) uniformly in \(u\). Compare with \eqref{9 April 2026 equation 3}. In \eqref{distributional identity} the symbol \((m)_\ell \coloneqq m(m-1) \cdots (m-\ell+1)\) denotes the falling factorial.
\end{enumerate}
In view of \eqref{root expansion} and \eqref{distributional identity}, upon careful examination of the powers of \(t\) and \(n\) arising from \eqref{Euler-Maclaurin first few terms}--\eqref{Euler-Maclaurin remainder term}, each contribution \(B(\lambda)\), \(T(\lambda)\), and \(I^{(j)}(\lambda)\) with \(j=0,1,3\), may be expressed as an asymptotic expansion in inverse powers of \(\lambda\) with remainder \(O(\lambda^{-4})\) as \(\lambda\to+\infty\). With some effort, it is also possible to show that
\begin{equation*}
\label{remainder term estimate}
	R^{(4)}(\lambda) = O(\lambda^{-4}) \qquad \text{as} \qquad \lambda \to +\infty\,.
\end{equation*}
Thus, taking into account formulae \eqref{4 January 2026 equation 9} and \eqref{4 January 2026 equation 10}, one eventually arrives at
\begin{equation}
\label{local counting function expansion Berger}
	\left(N_\pm'\ast{\mu}\right)(y;\lambda) = \frac{1}{2\pi^2}\,\lambda^2 - \frac{4-a^2}{6\pi^2} - \frac{2(a^2-1)^2}{15\pi^2}\,\lambda^{-2} \mp \frac{4a(a^2-1)^2}{15\pi^2}\,\lambda^{-3} + O(\lambda^{-4})\,.
\end{equation}
Since we have
\begin{subequations}
\begin{equation}
	\Sc(y) = 8 - 4a^2\,, \label{Berger scalar curvature}
	\end{equation}
	\begin{equation}
	\|\oRic(y)\|^2  = \frac{32}{3}(a^2-1)^2\,, \label{Berger traceless Ricci squared}
		\end{equation}
	\begin{equation}
	E_\alpha{}^{\beta\gamma}(y)\,\oRic^{\alpha\mu}(y)\,\nabla_\beta\,\oRic_{\gamma\mu}(y) = 32a(a^2-1)^2\,,\label{Berger pseudoscalar invariant}
\end{equation}
\end{subequations}
formula \eqref{local counting function expansion Berger} is consistent with the results from Sections~\ref{The first two Weyl coefficients}, \ref{The third Weyl coefficient}, and \ref{Higher Weyl coefficients}.

\begin{remark}
	\label{eta function Berger sphere via Hitchin's trick}
	It is possible to compute the residue of the first pole at $s=-2$ of the \emph{global} eta function of \(\curl\) on the Berger sphere in an alternative way, arguing as in \cite[subappendix~D.2]{curl}. Indeed, the eta function of \(\curl\) on the Berger sphere admits the representation
	\begin{equation*}
	\label{eta representation}
	\eta(s) = \theta(s) + (2a)^{-s} + 4a^s\zeta(s-1)\,,
	\end{equation*}
	where \(\zeta(s) = \sum_{n=1}^{\infty} n^{-s}\) is the Riemann zeta function, 
\begin{equation*}
\label{eta curl dima 2}
\theta(s)
:=
\sum_{j=1}^\infty
\left[
\left(
\sqrt{a^2+\mu_j}
\,+\,a
\right)^{-s}
-
\left(
\sqrt{a^2+\mu_j}
\,-\,a
\right)^{-s}
\right],
\end{equation*}
and the $\mu_j$ are the eigenvalues of $\,-\Delta\,$ enumerated in increasing order with account of multiplicities.	 Since \(\zeta(s)\) only exhibits a pole at \(s=1\), it suffices to compute the residue of \(\theta(s)\) at \(s=-2\). Borrowing ideas form arguments due to Hitchin, see~\cite[p.~34]{hitchin}, one derives the following representation for \(\theta(s)\):
	\begin{multline}
	\label{theta expansion}
	\theta(s) = -2s\,a\,\zeta_{\sqrt{-\Delta}}(s+1) + \frac{s(1-s^2)}{3}\, a^3 \,\zeta_{\sqrt{-\Delta}}(s+3)
	\\- \frac{s(1-s^2)(9-s^2)}{60} \,a^5\,\zeta_{\sqrt{-\Delta}}(s+5) + H(s)\,,
	\end{multline}
	where \(\zeta_{\sqrt{-\Delta}}(s)\) is the operator zeta function of \(\sqrt{-\Delta}\) and the remainder \(H(s)\) is holomorphic in the half-plane \(\operatorname{Re}(s)>-4\).
	One can show that \(\zeta_{\sqrt{-\Delta}}(s)\) admits, on a closed \(3\)-manifold, the expansion
	\begin{multline}
	\label{7 March 2026 equation 24}
	\zeta_{\sqrt{-\Delta}}(s) = \frac{1}{2\pi^2}\,\Bigg[\frac{\operatorname{Vol} M}{s-3}+\frac{1}{12(s-1)}\int_M\Sc(x)\,\rho(x)\,\dr x
	\\ -\frac{1}{1440(s+1)}\int_M\left[5\Sc(x)^2 + 6\,\|\oRic(x)\|^2\right]\rho(x)\,\dr x \Bigg] + K(s)\,,
	\end{multline}
	with \(K(s)\) holomorphic in \(\operatorname{Re}(s)>-2\). Using formulae \eqref{Berger scalar curvature} and \eqref{Berger traceless Ricci squared}, one may specialise \eqref{7 March 2026 equation 24} to the Berger sphere. Plugging the resulting formula into \eqref{theta expansion} and computing the residue at \(s=-2\) yields
	\begin{equation*}
	\label{global eta residue -2}
		\Res(\eta,-2) = -\frac{16}{15}a^2(a^2-1)^2\,,
	\end{equation*}
	which agrees with the corresponding formula~\eqref{local eta residue -2} from Section~\ref{Statement of the problem and main results}.
\end{remark}

\section*{Acknowledgements}
\addcontentsline{toc}{section}{Acknowledgements}

GB was supported by a UCL Research Excellence Scholarship.
MC was supported by EPSRC Fellowship EP/X01021X/1. MC is a member of the
GNAMPA research group of Indam, the Italian National Institute for Higher Mathematics. The
authors would like to thank the Isaac Newton Institute for Mathematical Sciences, Cambridge, for
support and hospitality during the programme Geometric Spectral Theory and Applications, where
work on this paper was undertaken. This work was supported by EPSRC grant EP/Z000580/1.

\begin{appendices}

\section{The amplitude-to-symbol operator}
\label{The amplitude-to-symbol operator}

In this appendix we present a concise account of the amplitude reduction procedure mentioned at the beginning of Step 6 of the algorithm from Section~\ref{The algorithm}.

\

Let $M$ be a $d$-dimensional manifold and
let $\tilde w \in C^\infty(\mathbb{R}\times M \times (T^*M\setminus \{0\}))$ be a polyhomogeneous function of order $r$,
\begin{equation*}
	\tilde w \sim \sum_{k=0}^{+\infty}\tilde w_{r-k}\,,
\end{equation*}
where each component $\tilde w_{r-k}$ is positively homogeneous in $\eta$ of degree $r-k$. We call $\tilde w$ an \emph{amplitude}.
Furthermore, let $\varphi \in C^\infty(\mathbb{R}\times M \times (T^*M\setminus \{0\}))$ be compatible with the flow $(x^*(t;y,\eta),\xi^*(t;y,\eta))$ associated with a given Hamiltonian function $h:T^*M\setminus\{0\}\to\mathbb{R}$ positively homogeneous of degree~1 in the momentum variable. Here `compatibility' is understood in the sense of Step 2 from the algorithm in Section~\ref{The algorithm}.
Finally, let
$\chi \in C^\infty(\mathbb{R}\times M \times (T^*M\setminus \{0\}))$
be a cut-off around the singularity $\varphi_\eta=0$ and away from the zero section.
Then the oscillatory integral
\begin{equation*}
	(2\pi)^{-d}\int e^{i\varphi(t,x;y,\eta)}\,\tilde w(t,x;y,\eta)\,\chi(t,x;y,\eta)\,\dr\eta
\end{equation*}
is equivalent, modulo an infinitely smooth contribution, to the oscillatory integral
\begin{equation*}
	(2\pi)^{-d}\int e^{i\varphi(t,x;y,\eta)}\,w(t;y,\eta)\,\chi(t,x;y,\eta)\,\dr\eta\,,
\end{equation*}
for some polyhomogeneous function $w \in C^\infty(\mathbb{R}\times (T^*M\setminus \{0\}))$ of order $r$.
We call $w$ the \emph{symbol}, more specifically, the \emph{right} symbol in view of its $x$-independence. The \emph{amplitude-to-symbol operator} is the linear operator
\begin{equation*}
	\mathfrak{S} : C^\infty(\mathbb{R}\times M\times (T^*M\setminus \{0\})) \to C^\infty(\mathbb{R}\times (T^*M\setminus \{0\}))
\end{equation*}
that maps $\tilde w \xmapsto{\mathfrak{S}} w$.
Following a procedure analogous to the one detailed in \cite[App.~A]{wave}, one shows that $\mathfrak{S}$ admits an asymptotic expansion in operators that decrease the degree of homogeneity by $k$,
\begin{equation}
\label{22 April 2026 equation 1}
	\mathfrak{S} \sim \sum_{k=0}^{+\infty}\mathfrak{S}_{-k}\,,
\end{equation}
where
\begin{equation}
\label{S zero}
	\mathfrak{S}_0 = \left.\left(\,\cdot\,\right)\right|_{x=x^*(t,y,\eta)}
\end{equation}
and
\begin{equation}
\label{S minus k}
	\mathfrak{S}_{-k} = \mathfrak{S}_0 \left[i\sum_{j=0}^{2k-1}\frac{(-1)^j}{j+1}\,\frac{\partial}{\partial\eta_\beta}\,F_j\,L_\beta\right]^k\,.
\end{equation}

In formula \eqref{S minus k}
\begin{enumerate}
\item 
each operator is understood to act on all  terms to right thereof,
\item
we make use of the operators
\begin{equation}
\label{operator L alpha}
	L_\beta
	:=
	\psi_\beta{}^\alpha\,
	\frac{\partial}{\partial x^\alpha}\,,
\end{equation}
where $\psi_\beta{}^\alpha(t,x;y,\eta)$ is the matrix inverse of $(\varphi_{x\eta})_\alpha{}^\beta(t,x;y,\eta) := \varphi_{x^\alpha\eta_\beta}(t,x;y,\eta)$, that is,
$(\varphi_{x\eta})_\alpha{}^\beta\,\psi_\beta{}^\gamma=\delta_\alpha{}^\gamma\,$, and
\item
the operators $F_j$ are defined in accordance with 
\begin{equation}
\label{operator F j}
	F_0 := 1\,, \qquad\qquad F_j := \frac{1}{j!}\,\frac{\partial\varphi}{\partial\eta_{\sigma_1}}\cdots\frac{\partial\varphi}{\partial\eta_{\sigma_j}}\,L_{\sigma_1}\cdots L_{\sigma_j}\,.
\end{equation}
Note that, although $\psi_\alpha{}^\beta$ depends on $x$, the operators \eqref{operator L alpha} commute, see \cite[Lemma~A.2]{wave}. In particular, there is no ordering ambiguity in the definition \eqref{operator F j} of the operators $F_j\,$.
\end{enumerate}

Formula \eqref{S minus k} provides an alternative representation of the amplitude-to-symbol operator, one that avoids reliance on multi-indices. The use of tensor index notation makes it more amenable to application in a geometric setting. Compare with \cite[formula~(A.5)]{wave} which also featured a weight, see Remark~\ref{remark weights}.

\begin{remark}
	The null Hamiltonian $h\equiv 0$ governs the constant `flow' $(x^*(t;y,\eta),\xi^*(t;y,\eta)) = (y,\eta)$. A distinguished phase function compatible with this flow reads
	\begin{equation*}
		\varphi(x;y,\eta) := \varphi(0,x;y,\eta) = (x-y)^\gamma\eta_\gamma\,,
	\end{equation*}
	with $x$ and $y$ `living' in the same coordinate chart. In this time-independent scenario, formulae \eqref{S zero} and \eqref{S minus k} amount to
	\begin{equation}
	\label{S zero time independent}
		\mathfrak{S}_0 = \left.\left(\,\cdot\,\right)\right|_{x=y}
	\end{equation}
	and
	\begin{equation}
	\label{S minus k time independent}
		\mathfrak{S}_{-k} = \frac{1}{k!}\,\mathfrak{S}_0 \left[i\,\frac{\partial^2}{\partial\eta_\beta\partial x^\beta}\right]^k\,.
	\end{equation}
	The operator \(\mathfrak{S}\sim\sum_{k\geq 0}\mathfrak{S}_{-k}\) with components \eqref{S zero time independent} and \eqref{S minus k time independent} reduces the amplitude of a pseudodifferential operator to a right symbol, see e.g.~\cite[Theorems~3.1~and~3.3]{shubin}.
\end{remark}

\section{Classification of scalar and pseudoscalar invariants}
\label{Scaling the Riemannian manifold}

Let us consider local Weyl coefficients $c_k^\pm(y)$, $k=2,1,0,-1,-2,\dots$,
for the operator $\curl$. The algorithm from Section~\ref{The algorithm} tells us that
these are expressed in terms of geometric invariants ---
metric, totally antisymmetric tensor, curvature and covariant derivatives of curvature  ---
including products and contractions of these geometric invariants.
The question at hand is which scalar or pseudoscalar geometric invariants can appear in the
explicit formula for $c_k^\pm(y)$ for a given $k$.
Recall that a pseudoscalar is a function $M\to\mathbb{R}$ which changes sign under orientation reversal
$\ast \mapsto - \ast$.

In order to answer this question we perform a scaling of the (covariant) metric tensor
\begin{equation}
\label{27 May 2026 equation 1}
g_{\alpha\beta}(x)
\mapsto
\omega^2\,
g_{\alpha\beta}(x)\,,
\end{equation}
where $\omega$ is a positive parameter.
The introduction of the scaling factor $\omega$ amounts, effectively,
to a change of unit of measurement.
It is easy to see that under the scaling \eqref{27 May 2026 equation 1}
our local counting functions transform as
\begin{equation}
\label{5 June 2026 equation 2}
N_\pm(y;\lambda)
\mapsto
\omega^{-3}\,
N_\pm(y;\omega\lambda)\,.
\end{equation}
Formula \eqref{5 June 2026 equation 2} induces the following transformation law for local Weyl coefficients:
\begin{equation}
\label{5 June 2026 equation 3}
c_k^\pm(y)
\mapsto
\omega^{k-2}\,
c_k^\pm(y)\,.
\end{equation}

\begin{definition}
\label{order of geometric invariant}
We say that a scalar or pseudoscalar geometric invariant is \emph{of order $p\in\mathbb{N}$}
if it acquires a factor
$\omega^{-p}$ under the scaling \eqref{27 May 2026 equation 1}.
\end{definition}

\begin{lemma}
\label{lemma linear combination}
Local Weyl coefficients $c_k^\pm(y)$
are linear combinations, with universal constants,
of scalar or pseudoscalar geometric invariants of degree $2-k$.
\end{lemma}
\begin{proof}
The claim follows at once from formula~\eqref{5 June 2026 equation 3}
and Definition~\ref{order of geometric invariant}.
\end{proof}

We now turn to a classification of scalar and pseudoscalar geometric invariants according to their behaviour under the rescaling~\eqref{27 May 2026 equation 1}. It is not hard to show that a geometric invariant containing $r$ instances of curvature and $q$ instances of $\nabla$ is of order $2r+q$ in the sense of Definition~\ref{order of geometric invariant} --- the number $q$ being even for scalars and odd for pseudoscalars. In light of this fact, one can produce a complete list of nonvanishing invariants of order at most five in dimension $d=3$:

\begin{longtable}{c c c c}
	\hline
	\\ [-1em]
	\multicolumn{1}{c}{\textbf{Order}} & 
	\multicolumn{1}{c}{\textbf{Scalar}} &
	\multicolumn{1}{c}{\textbf{Pseudoscalar}} &
	\multicolumn{1}{c}{\textbf{Weyl coefficient}} \\ \\ [-1em]
	\hline \hline \\ [-1em]
	$0$ & $1$ & --- & $c_2^\pm(y)$ \\[.3em] \\ [-1em]
	$1$ & --- & --- & $c_1^\pm(y)$ \\[.3em] \\ [-1em]
	$2$ & $\Sc(y)$ & --- & $c_0^\pm(y)$ \\[.3em] \\ [-1em]
	$3$ & --- & --- & $c_{-1}^\pm(y)$ \\[.3em] \\ [-1em]
	$4$ & $\Sc^2(y)\,,$ \ $\|{\oRic}(y)\|^2\,,$ \ $(\Delta\Sc)(y)$ \ & --- & $c_{-2}^\pm(y)$ \\[.3em] \\ [-1em]
	$5$ & --- & $E_\alpha{}^{\beta\gamma}(y)\,\oRic^{\alpha\mu}(y)\,\nabla_\beta\,\oRic_{\gamma\mu}(y)$ & $c_{-3}^\pm(y)$ \\[.3em] \\ [-1em]
	\hline
\end{longtable}

Observe that scalar and pseudoscalar geometric invariants of a particular order form finite-dimensional real vector spaces. What is given in the table above are, effectively, bases for such vector spaces.

\begin{remark}
One might think that, when dealing with scalar invariants of order 4, we missed a fourth basis element, namely,
$\nabla_\alpha\nabla_\beta\,\oRic^{\alpha\beta}$.
However, the contracted Bianchi identity $\nabla_\alpha\Ric^\alpha{}_\beta=\frac{1}{2}\nabla_\beta\Sc$
allows one to express
$\nabla_\alpha\nabla_\beta\,\oRic^{\alpha\beta}$
via
$\Delta\Sc\,$.
\end{remark}

\begin{remark}
	The above classification bears similarity with that presented in~\cite[Section~2.4]{gilkeyheat}. 
	Because \(\curl\) feels orientation, in this paper we are interested in the behaviour of geometric invariants under orientation reversal in addition to metric rescaling. In \cite{gilkeyheat,branson}, pseudoscalars are handled in a somewhat implicit manner, making their role less immediately clear. 
	Observe that the notion of order of a geometric invariant in the sense of \cite[Section~2.4]{gilkeyheat} agrees with that of Definition~\ref{order of geometric invariant}.
\end{remark}

\section{Symmetries of Weyl coefficients}
\label{appendix symmetry of Weyl coefficients}

In this appendix we provide a proof of Theorem~\ref{theorem symmetries of Weyl coefficients}.

\begin{proof}[Proof of Theorem~\ref{theorem symmetries of Weyl coefficients}]

Consider coefficients $c_k^\pm(y)$ with $k$ even. With account of Lemma~\ref{lemma linear combination} and the observations immediately thereafter, one concludes that $c_k^\pm(y)$ is a linear combination of scalar invariants. 
Analogously, coefficients $c_k^\pm(y)$ with $k$ odd are linear combinations of pseudoscalar invariants. 
These facts imply~\eqref{Symmetries of Weyl coefficients equation}, as soon as one observes that under orientation reversal one has
\[
N_\pm(y;\lambda)\mapsto N_\mp(y;\lambda)\,.\qedhere
\] 
\end{proof}

\end{appendices}


\begin{thebibliography}{42}
\addcontentsline{toc}{section}{References}


\bibitem{asymm1}
M.F.~Atiyah, V.K.~Patodi and I.M.~Singer, 
Spectral asymmetry and Riemannian geometry,
{\it Bull.~Lond.~Math.~Soc.} \textbf{5} (1973) 229--234.
DOI: \href{https://doi.org/10.1112/blms/5.2.229}{10.1112/blms/5.2.229}.

\bibitem{asymm2}
M.F.~Atiyah, V.K.~Patodi and I.M.~Singer, 
Spectral asymmetry and Riemannian geometry I,
{\it Math.~Proc.~Camb.~Philos.~Soc.} \textbf{77} (1975) 43--69.
DOI: \href{https://doi.org/10.1017/S0305004100049410}{10.1017/S0305004100049410}.

\bibitem{asymm3}
M.F.~Atiyah, V.K.~Patodi and I.M.~Singer, 
Spectral asymmetry and Riemannian geometry II,
{\it Math.~Proc.~Camb.~Philos.~Soc.} \textbf{78} (1975) 405--432.
DOI: \href{https://doi.org/10.1017/S0305004100051872}{10.1017/S0305004100051872}.

\bibitem{asymm4}
M.F.~Atiyah, V.K.~Patodi and I.M.~Singer, 
Spectral asymmetry and Riemannian geometry III,
{\it Math.~Proc.~Camb.~Philos.~Soc.} \textbf{79} (1976) 71--99.
DOI: \href{https://doi.org/10.1017/S0305004100052105}{10.1017/S0305004100052105}.

%

%
%
%
%
%


\bibitem{baer_curl}
C.~B\"ar, 
The curl operator on odd-dimensional manifolds,
{\it J.~Math.~Phys.} {\bf 60} (2019) 031501.
DOI: \href{https://doi.org/10.1063/1.5082528}{10.1063/1.5082528}.

%
%

%
%

\bibitem{birman_curl}
M.~Sh.~Birman and M.~Z.~Solomyak,
The Weyl asymptotics of the spectrum of the Maxwell operator for domains with a Lipschitz boundary, 
{\it Vestn.~Leningr.~Univ.~Mat.} {\bf 20} no.~3 (1987) 15--21.

\bibitem{birman_curl2}
M.~Sh.~Birman and M.~Z.~Solomyak,
$L^2$-Theory of the Maxwell operator in arbitrary domains, 
{\it Uspekhi Mat.~Nauk} {\bf 42} no.~6 (1987) 61--76; {\it Russian Math. Surveys} {\bf 42} no.~6 (1987) 75--96.
DOI: \href{https://doi.org/10.1070/RM1987v042n06ABEH001505}{10.1070/RM1987v042n06ABEH001505}.



\bibitem{branson}
T.~P.~Branson and P.~B.~Gilkey, Residues of the eta function for an operator
of Dirac type,
{\it J.~Funct.~Anal.} \textbf{108} (1992) 47--87.
DOI: \href{https://doi.org/10.1016/0022-1236(92)90146-A}{10.1016/0022-1236(92)90146-A}.


%
%

\bibitem{diagonalisation}
M.~Capoferri,
Diagonalization of elliptic systems via pseudodifferential projections,
{\it J.~Differ.~Equ.} {\bf 313} (2022) 157--187.
DOI: \href{https://doi.org/10.1016/j.jde.2021.12.032}{10.1016/j.jde.2021.12.032}.

\bibitem{dirac_asymmetry}
M.~Capoferri, B.~Costeri, C.~Dappiaggi,
Spectral asymmetry via pseudodifferential projections: the massless Dirac operator,
{\it Anal.~Math.~Phys.} {\bf 16} no.~74 (2026).
DOI: \href{https://doi.org/10.1007/s13324-026-01210-w}{10.1007/s13324-026-01210-w}.


\bibitem{wave}
M.~Capoferri, M.~Levitin and D.~Vassiliev,
Geometric wave propagator on Riemannian manifolds,
\emph{Comm.~Anal.~Geom.} {\bf 30} no.~8 (2022) 1713--1777.
DOI: \href{https://doi.org/10.4310/CAG.2022.v30.n8.a2}{10.4310/CAG.2022.v30.n8.a2}.

\bibitem{obstructions}
M.~Capoferri, G.~Rozenblum, N.~Saveliev and D.~Vassiliev,
Topological obstructions to the diagonalisation of pseudodifferential systems,
{\it Proc.~Amer.~Math.~Soc.~(Ser.~B)} {\bf 9} (2022) 472--486.
DOI: \href{https://doi.org/10.1090/bproc/147}{10.1090/bproc/147}.



\bibitem{dirac}
M.~Capoferri and D.~Vassiliev,
Global propagator for the massless Dirac operator and spectral asymptotics,
{\it Integral Equ.~Oper.~Theory} {\bf 94} (2022) 30.
DOI: \href{https://doi.org/10.1007/s00020-022-02708-1}{10.1007/s00020-022-02708-1}.

\bibitem{part1}
M.~Capoferri and D.~Vassiliev,
Invariant subspaces of elliptic systems I: pseudodifferential projections,
{\it J.~Funct.~Anal.} {\bf 282} no.~8 (2022) 109402.
DOI: \href{https://doi.org/10.1016/j.jfa.2022.109402}{10.1016/j.jfa.2022.109402}.

\bibitem{part2}
M.~Capoferri and D.~Vassiliev,
Invariant subspaces of elliptic systems II: spectral theory,
{\it J.~Spectr.~Theory} {\bf 12} no.~1 (2022) 301--338.
DOI: \href{https://doi.org/10.4171/JST/402}{10.4171/JST/402}.

\bibitem{curl}
M.~Capoferri and D.~Vassiliev,
Beyond the Hodge theorem: curl and asymmetric pseudodifferential projections,
{\it J.~London Math.~Soc.}
{\bf 113} no.~1 (2026) e70431.
DOI: \href{https://doi.org/10.1112/jlms.70431}{10.1112/jlms.70431}.

\bibitem{conjectures}
M.~Capoferri and D.~Vassiliev, 
A microlocal pathway to spectral asymmetry: curl and the eta
invariant. Preprint \href{https://doi.org/10.48550/arXiv.2502.18307}{arXiv:2502.18307} (2025). 

\bibitem{CDV}
O.~Chervova, R.~J.~Downes and D.~Vassiliev,
The spectral function of a first order elliptic system,
{\it J.~Spectr.~Theory} \textbf{3} no.~3 (2013) 317--360. 
DOI: \href{https://doi.org/10.4171/JST/47}{10.4171/JST/47}.

%
%
%


\bibitem{filonov_curl1}
M.N.~Demchenko and N.~Filonov, 
Spectral asymptotics of the Maxwell operator on Lipschitz manifolds with boundary, in: \emph{Spectral Theory of Differential Operators: M.Sh.~Birman 80th Anniversary Collection} (Amer.~Math.~Soc., Providence, RI, 2008), 73--90.
DOI: \href{https://doi.org/10.1090/trans2/225}{10.1090/trans2/225}.



\bibitem{DuGu}
J.~J.~Duistermaat and V.~W.~Guillemin,
The spectrum of positive elliptic operators and periodic bicharacteristics,
{\it Invent.~Math.} \textbf{29} no.~1 (1975) 39--79.
DOI: \href{https://doi.org/10.1007/BF01405172}{10.1007/BF01405172}.


\bibitem{peralta2}
A.~Enciso and D.~Peralta-Salas,
Non-existence of axisymmetric optimal domains with smooth boundary for the first curl eigenvalue, {\it Ann.~Sc.~Norm.~Super.~Pisa Cl.~Sci.} {\bf XXIV} (2023) 311--327.
DOI: \href{https://doi.org/10.2422/2036-2145.202010_008}{10.2422/2036-2145.202010\_008}.

\bibitem{peralta1}
A.~Enciso, W.~Gerner and D.~Peralta-Salas,
Optimal convex domains for the first curl eigenvalue,
{\it Trans.~Am.~Math.~Soc.}, {\bf 377} no.~7 (2024) 4519--4540.
DOI: \href{https://doi.org/10.1090/tran/8914}{10.1090/tran/8914}.

\bibitem{peralta3}
A.~Enciso, W.~Gerner and D.~Peralta-Salas,
Optimal metrics for the first curl eigenvalue on $3$-manifolds,
{\it Calc.~Var.~Partial Differ.~Equ.} {\bf 64} (2025) 146.
DOI: \href{https://doi.org/10.1007/s00526-025-02995-7}{10.1007/s00526-025-02995-7}.


\bibitem{fang_PhD}
Y.-L. Fang,
{\it Analysis of first order systems on manifolds without boundary: A spectral theoretic approach}.
PhD Thesis, University College London, 2017.
URL: \url{https://discovery.ucl.ac.uk/id/eprint/1560979}.


\bibitem{filonov_curl3}
N.~Filonov,
Spectral analysis of the selfadjoint operator curl in a region of finite measure,
{\it Algebra i Anal.} {\bf 11} no.~6 (1999) 178--190; 
{\it St.~Petersburg Math.~J.} {\bf 11} no.~6 (2000) 1085--1095.

\bibitem{filonov_curl2}
N.~Filonov, 
Weyl asymptotics of the spectrum of the Maxwell operator in Lipschitz domains of arbitrary dimension, 
{\it Algebra i Anal.} {\bf 25} no.~1 (2013) 170--215;
{\it St.~Petersburg Math.~J.} {\bf 25} no.~1 (2014) 117--149.
DOI: \href{https://doi.org/10.1090/S1061-0022-2013-01282-9}{10.1090/S1061-0022-2013-01282-9}.



\bibitem{giga}
Z.~Yoshida and Y.~Giga, 
Remarks on spectra of operator rot,
{\it Math.~Z.} {\bf 204} (1990) 235--245.
DOI: \href{https://doi.org/10.1007/BF02570870}{10.1007/BF02570870}.


\bibitem{gilkeyheat}
P.~B.~Gilkey, 
{\it Invariance theory, the heat equation, and the Atiyah--Singer index theorem}, 2\textsuperscript{nd}~ed., Studies in Advanced Mathematics,
CRC Press, 1994.

%



\bibitem{hitchin}
N.~Hitchin,
Harmonic spinors,
{\it Adv.~Math.} {\bf 14} no.~1 (1974) 1--55.
DOI: \href{https://doi.org/10.1016/0001-8708(74)90021-8}{10.1016/0001-8708(74)90021-8}.


%

\bibitem{Ivr80}
V.~Ivrii,
Second term of the spectral asymptotic expansion of the Laplace--Beltrami operator on manifolds with boundary,
{\it Funct.~Anal.~Appl.} \textbf{14} (1980) 98--106.
DOI: \href{https://doi.org/10.1007/BF01086550}{10.1007/BF01086550}.


\bibitem{Ivr84}
V.~Ivrii,
{\it Precise spectral asymptotics for elliptic operators
	acting in fiberings over manifolds with boundary},
Lecture Notes in Mathematics \textbf{1100}, Springer-Verlag, Berlin, 1984.
DOI: \href{https://doi.org/10.1007/BFb0072205}{10.1007/BFb0072205}.

\bibitem{Ivr98}
V.~Ivrii,
{\it Microlocal analysis and precise spectral asymptotics},
Springer-Verlag, Berlin, 1998.
DOI: \href{https://doi.org/10.1007/978-3-662-12496-3}{10.1007/978-3-662-12496-3}.

%

%
%


\bibitem{LSV}
A.~Laptev, Yu.~Safarov and D.~Vassiliev,
On global representation of Lagrangian distributions and solutions of hyperbolic equations,
{\it Commun.~Pure~Appl. Math.} \textbf{47}  no.~11 (1994) 1411--1456.
DOI: \href{https://doi.org/10.1002/cpa.3160471102}{10.1002/cpa.3160471102}.




\bibitem{lerner}
N.~Lerner and F.~Vigneron,
On some properties of the curl operator and their consequences for the Navier-Stokes system, {\it Commun.~Math.~Res.} {\bf 38} no.~4 (2022) 449--497.
DOI: \href{https://doi.org/10.4208/cmr.2021-0106}{10.4208/cmr.2021-0106}.


%
%
%

\bibitem{lotay}
J.~Lotay,
Stability of coassociative conical singularities,
\emph{Comm.~Anal.~Geom.} {\bf 20} no.~4 (2012) 803--867.
DOI: \href{https://doi.org/10.4310/CAG.2012.v20.n4.a5}{10.4310/CAG.2012.v20.n4.a5}.


%
%

\bibitem{Uwe Muller}
U.~M\"uller, C.~Schubert and A.E.M.~van de Ven,
A closed formula for the Riemann normal coordinate expansion,
{\it Gen.~Relativ.~Gravit.} \textbf{31} (1999) 1759--1768.
DOI: \href{https://doi.org/10.1023/A:1026718301634}{10.1023/A:1026718301634}.

%
%
%

\bibitem{rellich}
F.~Rellich,
{\it Perturbation theory of eigenvalue problems},
Courant Institute of Mathematical Sciences, New York University, 1954.

\bibitem{safarov_curl}
Yu.~Safarov,
Asymptotic behavior of the spectrum of the Maxwell operator, 
{\it J.~Sov.~Math.} {\bf 27} (1984) 2655--2661.
DOI: \href{https://doi.org/10.1007/BF01103726}{10.1007/BF01103726}.


\bibitem{SaVa}
Yu.~Safarov and D.~Vassiliev,
{\it The asymptotic distribution of eigenvalues of partial differential operators},
Am.~Math.~Soc., Providence (RI), 1997.
DOI: \href{https://doi.org/10.1090/mmono/155}{10.1090/mmono/155}

%

%
%
%

\bibitem{shubin}
M.A.~Shubin,
{\it Pseudodifferential operators and spectral theory},
Springer, 2001.
DOI: \href{https://doi.org/10.1007/978-3-642-56579-3}{10.1007/978-3-642-56579-3}.

%

%
%

%
%
%
%

\end{thebibliography}
\end{document}